\documentclass[a4paper]{amsart}

\usepackage[margin=3cm,top=2.5cm,bottom=2cm]{geometry}

\usepackage{amsmath, amssymb, amsthm, mathrsfs}
\usepackage{graphicx,enumitem,microtype} 
\usepackage[dvipsnames]{xcolor}

\usepackage[colorlinks=true,linkcolor=red!70!black,urlcolor=MidnightBlue,citecolor=MidnightBlue]{hyperref}

\usepackage[utf8]{inputenc} 
\usepackage[T1]{fontenc}
\usepackage{lmodern}

\newcommand\N{{\mathbb N}}
\newcommand\R{{\mathbb R}}

\def\AA{{\mathcal A}}
\def\BB{{\mathcal B}}

\def\DD{{\mathcal D}}

\def\HH{{\mathcal H}}

\def\OO{{\mathcal O}}

\def\TT{{\mathcal T}}
\def\UU{{\mathcal U}}
\def\VV{{\mathcal V}}

\def\VV{{\mathcal V}}

\def\CCC{{\mathscr C}}
\def\DDD{{\mathscr D}}

\def\LLL{{\mathscr L}}
\def\MMM{{\mathscr M}}

\def\RRR{{\mathscr R}}
\def\SSS{{\mathscr S}}

\def\eps{{\varepsilon}}

\newcommand{\la}{\langle}
\newcommand{\ra}{\rangle}

\def\Lloc{L_{\mathrm{loc}}} 
\newcommand{\Nt}{|\hskip-0.04cm|\hskip-0.04cm|}

\DeclareMathOperator{\Div}{div}

\newtheorem{theo}{Theorem}[section]
\newtheorem{prop}[theo]{Proposition}
\newtheorem{lem}[theo]{Lemma}

\newtheorem*{thm*}{Theorem}
\theoremstyle{remark}
\newtheorem{rem}[theo]{Remark}

\newtheorem*{ex*}{Example}
\theoremstyle{definition}

\numberwithin{equation}{section}
\newcommand{\be}{\begin{equation}}
\newcommand{\ee}{\end{equation}}
\newcommand{\ba}{\begin{aligned}}
\newcommand{\ea}{\end{aligned}}
\newcommand{\beqn}{\begin{equation}}
\newcommand{\eeqn}{\end{equation}}
\newcommand{\bear}{\begin{eqnarray}}
\newcommand{\eear}{\end{eqnarray}}
\newcommand{\bean}{\begin{eqnarray*}}
\newcommand{\eean}{\end{eqnarray*}}
\newcommand{\bal}{\begin{aligned}}
\newcommand{\eal}{\end{aligned}}

\title[Kinetic Fokker-Planck equation in a domain]{The Kinetic Fokker-Planck equation in a domain:  \\  Ultracontractivity, hypocoercivity \\ and long-time asymptotic  behavior}
\author[K. Carrapatoso]{Kleber Carrapatoso}
\author[S. Mischler]{St\'ephane Mischler}

\address[K.~Carrapatoso]{Centre de Math\'ematiques Laurent Schwartz, \'Ecole polytechnique, Institut Polytechnique de Paris, 91128 Palaiseau cedex, France}
\email{kleber.carrapatoso@polytechnique.edu}

\address[S.~Mischler]{Centre de Recherche en Math\'ematiques de
  la D\'ecision (CEREMADE),
  Universit\'es PSL \& Paris-Dauphine, Place de Lattre de
  Tassigny, 75775 Paris 16, France \& Institut Universitaire de France (IUF)}
\email{mischler@ceremade.dauphine.fr}

\date{\today}



\begin{document}

\begin{abstract} 
We consider the Kinetic Fokker-Planck (FKP) equation in a domain with Maxwell reflection condition on  the boundary. We establish the ultracontractivity of the associated  semigroup and the hypocoercivity of the   associated  operator. We deduce the convergence with constructive rate of the solution to the KFP equation towards the stationary state with same mass as the initial datum.
\end{abstract}

\maketitle

\tableofcontents
 
\section{Introduction}

In this paper, we consider the Kinetic Fokker-Planck (KFP) equation, also called the degenerated Kolmogorov or the ultraparabolic equation, 
\beqn\label{eq:Kolmogorov} 
\partial_t f + v \cdot \nabla_x f - \Delta_v f - \Div_v (vf) = 0 \quad \hbox{in}\quad \UU
\eeqn
on the function $f := f_t = f(t,\cdot) = f(t,x,v)$, with $(t,x,v) \in \UU := (0,T) \times \Omega \times \R^d$, $T \in (0,+\infty]$, $\Omega \subset \R^d$ a suitably smooth domain, $d \ge 3$, 
complemented with the Maxwell reflection condition on  the boundary 
\beqn\label{eq:KolmoBdyCond}
\gamma_{-} f  = \RRR \gamma_{+}  f  =(1-\iota) \SSS  \gamma_{+}  f + \iota \DDD  \gamma_{+}  f     \quad\hbox{on}\quad \Gamma_-,
\eeqn
and associated to an initial condition 
\beqn\label{eq:initialDatum} 
f(0,x,v) = f_0(x,v) \quad \hbox{in}\quad \OO := \Omega \times \R^d.
\eeqn
Here $\Gamma_-$ denotes the  incoming part of the boundary, $\SSS$ denotes the specular reflection operator, $\DDD$ denotes the diffusive reflection operator  (see precise definitions below), and 
$\iota : \partial\Omega \to [0,1]$ denotes a (possibly space dependent) accommodation coefficient. More precisely, we assume that $\Omega := \{ x \in \R^d; \,   \delta(x) > 0 \}$  
for a $W^{2,\infty}(\R^d)$ function $\delta$ such that $|\delta (x)| := \mathrm{dist}(x,\partial\Omega)$ on a neighborhood of  the boundary set $\partial\Omega$ and thus $n_x = n(x) :=  - \nabla \delta(x)$
coincides with the  unit normal outward vector field on $\partial\Omega$. We next define $\Sigma_\pm^x := \{ {v} \in \R^d; \pm \, {v} \cdot n_x > 0 \}$ the sets of outgoing ($\Sigma_+^x$) and incoming ($\Sigma_-^x$) velocities at point $x \in \partial\Omega$, then the sets
$$
 \Sigma_\pm := 
 \{ (x,{v}); \, x \in \partial\Omega, \, {v} \in \Sigma^x_\pm \}, 
 \quad
 \Gamma_\pm := (0,T) \times \Sigma_\pm,
$$
and finally   the outgoing and incoming trace functions $\gamma_\pm f := {\bf 1}_{\Gamma_\pm}  \gamma f$.
The specular reflection operator $\SSS$ is defined by 
\beqn
\label{eq:FPK-def_Gamma}
(\SSS g) (x,v)  :=  g (x , \VV_x v), \quad \VV_x v := v - 2 n_x (n_x \cdot v),
\eeqn
and the diffusive operator $\DDD$ is defined by 
\beqn
\label{eq:FPK-def_D}
(\DDD g) (x,v)  := \MMM(v) \widetilde g (x), \quad \widetilde g (x) := \int_{\Sigma^x_+} g(x,w) \, (n_x \cdot w) \, d w,
\eeqn
where $\MMM$ stands for the (conveniently normalized) Maxwellian function
\beqn
\label{eq:FPK-def_M}
\MMM(v) := (2\pi)^{-(d-1)/2} \exp(-|v|^2/2),
\eeqn
which is positive on $\R^d$ and verifies $\widetilde \MMM=1$. We assume that  the accommodation coefficient satisfies $\iota \in W^{1,\infty}(\partial\Omega)$. 
For further references, we also define  the (differently normalized) Maxwellian function
\beqn\label{eq:FPK-def_mu} 
f_\infty (x,v) = \frac{1}{|\Omega|}\, \mu(v) := \frac{1 }{ |\Omega| (2\pi)^{d/2}} \exp(-|v|^2/2),
\eeqn
which is positive on $\OO$ and verifies $\| f_\infty \|_{L^1(\OO)} = 1$.  The elementary  (and well known at least at a formal level) properties of the Kinetic Fokker-Planck equation 
are that it is mass conservative, namely
\beqn
\label{eq:intro_mass_cons}
\langle\! \langle f_t \rangle\!\rangle =  \langle\! \langle f_0 \rangle\!\rangle , \quad \forall \, t \ge 0, \quad\hbox{with}\quad \langle\! \langle  h  \rangle\!\rangle := \int_\OO h dxdv, 
\eeqn
it is positivity  preserving, namely $f_t \ge 0$ if $f_0 \ge 0$, and $f_\infty$ is a stationary solution.

\medskip

The aim of this paper is twofold:

\smallskip
(1) On the one hand, we prove the ultracontractivity of the semigroup associated to the evolution problem
\eqref{eq:Kolmogorov}--\eqref{eq:KolmoBdyCond}--\eqref{eq:initialDatum} by establishing some immediate gain of Lebesgue integrability and even  immediate uniform bound estimate. 

\smallskip
(2) On the other hand, we prove the convergence of the solution to the associated stationary state, namely $f_t \to  \langle\! \langle f_0 \rangle\!\rangle f_\infty$ as $t \to \infty$, with 
constructive exponential rate in many weighted Lebesgue spaces. 

\smallskip
These results extend some previous similar results known for other geometries or less general reflection conditions. 
 For both problems, we adapt or modify some recent   or forthcoming results established  in \cite{MR4581432,CM-Landau**} for the Landau equation for the same geometry as considered here. 
In that sense, the techniques are not really new and the present contribution may rather be seen as a pedagogical illustration on one of the simplest models of the kinetic theory of some tools we develop in other papers for more elaborated kinetic models. We also refer to   \cite{CGMM**,FM**,CM-Boltz**} for further developments of these techniques for related 
kinetic equations set in a domain with reflection conditions on the boundary.

\smallskip 
For a weight function $\omega : \R^d \to (0,\infty)$ and a exponent $p \in [1,\infty]$, we define the associated weighted Lebesgue space  
$$
L^p_\omega := \{  f \in \Lloc^1 ({\R}^d); \,\, \|  f \|_{L^p_\omega}  := \|  f \omega \|_{L^p} < \infty \}. 
$$

\medskip
Our first main result is an ultracontractivity property. 

\begin{theo}\label{th:DGNML1Linfty}  
There exist two weight functions $\omega = \exp (\zeta |v|^2)$, $\omega' = \exp (\zeta' |v|^2)$, with $0 < \zeta' < \zeta < 1/2$, and some constants $\nu > 0$, $C_1 \ge 1$, $C_2 \ge 0$ such that for any exponents $p,q \in [1,\infty]$, $q>p$, 
and any initial datum $f_0 \in L^p_\omega(\OO)$, 
 the   associated solution $f$ to the Kinetic Fokker-Planck (KFP) equation \eqref{eq:Kolmogorov}--\eqref{eq:KolmoBdyCond}--\eqref{eq:initialDatum} satisfies    
\beqn\label{eq:th1-estim}
\| f (t) \|_{L^q_{\omega'}} \le C_1 \frac{e^{C_2 t} }{ t^{\nu(1/p-1/q)}}   \| f_0  \|_{L^p_\omega}, \quad \forall \, t > 0.
\eeqn

\end{theo}

We refer to section~\ref{subsec:conclusion} for a possible definition of the set $\mathfrak W_1$ of weight functions $\omega$ for which the above  ultracontractivity property holds true.
In the whole space $\Omega=\R^d$, such a kind of  ultracontractivity property is a direct consequence of the representation of the solution thanks to the Kolmogorov kernel, see \cite{MR1503147}, as well as \cite{MR0222474,MR1949176} for related regularity estimates.
 Some local uniform estimate of a similar kind for a larger class of KFP equations in the whole space has been established \cite{MR2068847,MR2368979,MR3981630} by using Moser iterative scheme introduced in \cite{MR0170091,MR159139}, from what some Gaussian upper bound on the fundamental solution may be derived,  see \cite{MR2000352,MR3717350,francesca2021fundamental}. 
 In  \cite{MR3923847}, the same local uniform estimates (as well as the Harnack inequality and the Holder regularity) has been shown for a still larger class of KFP equations in the whole space by using  De Giorgi iterative scheme as introduced in  \cite{MR0093649}.  
 We also refer  to  \cite{anceschi2019survey} for a general survey about these issues and to \cite{MR2530175,MR2773175,MR4181953,MR4491355,MR4386001} for additional results on the KFP equations in the whole space.
In \cite{Herau2007}, a gain of regularity estimate has been established by adapting Nash argument introduced in \cite{MR0100158}, see also \cite{MR2562709,MR3779780,MR3488535} for further developments of the same technique. 

 In \cite{MR4414615}, an  ultracontractivity result similar to  ours is obtained for the KFP equation in a domain with specular reflection at the boundary by an extension argument to the whole space (used first in \cite{MR4076068}) and then reduces the problem to the application of \cite{MR2068847,MR3923847}. 
 In \cite{zhu2022regularity} some kind of regularity up to the boundary is proved for the KFP equation with inflow or specular reflection at the boundary using the extension argument of  \cite{MR4414615} and some appropriate change of coordinates. See also \cite{MR4527757}, where some similar results are established for  the KFP equation with zero inflow. We finally refer to  \cite{CGMM**} where the same kind of ultracontractivity result is established with $\omega' = \omega$ for a large class of weight functions $\omega$.

\medskip
We are next concerned with the longtime behavior estimate. We start by establishing a hypocoercivity result. For that purpose, we define the operator 
\beqn\label{def:LLL}
\LLL f := -  v \cdot \nabla_x f + \Delta_v f + \Div_v (vf) 
\eeqn
and we denote by $ \mathrm{Dom}(\LLL)$ its domain in   the Hilbert space $\HH := L^2(\mu^{-1}dxdv)$ endowed with the norm $\| f \|_{\HH} = \| \mu^{-1/2} f \|_{L^2}$. 

\begin{theo}\label{theo:hypo}
There exists a scalar product $(\!( \cdot , \cdot )\!)$ on the space $\HH$ so that the associated norm $\Nt \cdot \Nt$ is equivalent to the usual norm $\| \cdot \|_{\HH}$, and for which the linear operator $\LLL$ satisfies the following coercivity estimate: there is a positive constant $\lambda \in (0,1)$ such that 
\beqn\label{eq:hypocoercivityL2}
( \! ( - \LLL f , f )\! ) \ge \lambda  \Nt f \Nt^2
\eeqn
for any $f \in \mathrm{Dom}(\LLL)$ satisfying the boundary condition \eqref{eq:KolmoBdyCond} and the mass condition $ \la \! \la f \ra\!\ra = 0$.
\end{theo}

The result and the proof is a mere adaptation and simplification of the same hypocoercivity estimate established in \cite{MR4581432}. This last one is inspired, generalizes and simplifies some previous results established in \cite{MR2679358,MR3562318}, see also \cite{MR1787105,MR1969727,MR2034753,MR2130405,Herau2006,MR2562709,MR3324910} and the references therein for more material about the hypocoercivity theory.

\smallskip

We deduce from the two previous results the announced exponential convergence result.  
 
\begin{theo}\label{th:LimitInfty} There exists a class of weight functions $\mathfrak W_2$ such that for any weight function $\omega \in \mathfrak W_2$, any exponent $p \in [1,\infty]$ and any initial datum $f_0 \in L^p_\omega(\OO)$, the  associated solution $f$ to the KFP equation \eqref{eq:Kolmogorov}--\eqref{eq:KolmoBdyCond}--\eqref{eq:initialDatum} satisfies    
\begin{equation}\label{eq:th:LimitInfty}
\| f (t) - \langle\!\langle f_0 \rangle\!\rangle f_\infty  \|_{L^p_\omega}   \le C e^{-\lambda t}   \|  f_0 - \langle\!\langle f_0 \rangle\!\rangle f_\infty   \|_{L^p_\omega}, \quad \forall \, t \ge 0,
\end{equation}
for the same constant $\lambda \in (0,1)$ as in Theorem~\ref{theo:hypo} and for some constant $C = C(\omega)$. 
\end{theo}
 
 It is worth emphasizing that the set $\mathfrak W_2$ contains some exponential functions and some polynomial (increasing fast enough) functions.
The case $p=2$ and $\omega = \mu^{-1/2}$ is an immediate consequence of Theorem~\ref{theo:hypo}. The general case is then deduced from this particular one thanks to Theorem~\ref{th:DGNML1Linfty} and some enlargement and shrinking techniques introduced and developed in \cite{MR3779780,MR3488535,MR3465438}.

\medskip
 Let us end the introduction by describing the organization of the paper which is mainly dedicated to the proof of the above results.
 
 \smallskip
 In Section~\ref{sec:WeightLp} we establish some growth estimates in many weighted Lebesgue spaces on the semigroup associated to  the KFP equation \eqref{eq:Kolmogorov}--\eqref{eq:KolmoBdyCond}--\eqref{eq:initialDatum}.  
 We do not discuss the existence and uniqueness issues about solutions  to  the KFP equation and the construction of the associated positive semigroup which will be discussed in detail in the companion paper \cite{CGMM**}. We however emphasizes that solutions to the KFP equation must be understood in the renormalized sense as defined in \cite{MR0972541,MR2721875} so that the associated trace functions are well defined, see \cite{MR2721875,CGMM**,CM-Landau**} and the references therein. We thus rather focus on the (a priori) estimates by exhibiting suitable twisted weight estimates for the solutions to the  KFP equation \eqref{eq:Kolmogorov}--\eqref{eq:KolmoBdyCond}--\eqref{eq:initialDatum} and its dual counterpart. 

 Section~\ref{sec:proofTh1} is dedicated to the proof of Theorem~\ref{th:DGNML1Linfty}. The strategy mixes Moser's gain of integrability argument of \cite{MR159139} and Nash's duality and interpolation arguments of  \cite{MR0100158}. 
It is also based on a twisted weight argument which is  somehow slightly more elaborated than the one used in the previous sections. 
In Section~\ref{sec:proofTh2}, we prove Theorem~\ref{theo:hypo}
and Section~\ref{sec:proofTh3}  is dedicated to the proof of Theorem~\ref{th:LimitInfty}

\section{Weighted $L^p$ growth estimates }
 \label{sec:WeightLp}
 
 This section is devoted to the proof of a first and somehow rough set of growth estimates in some convenient weighted $L^p$ spaces for solutions
 to the KFP equation \eqref{eq:Kolmogorov}--\eqref{eq:KolmoBdyCond}--\eqref{eq:initialDatum} and the associated semigroup that we denote by the same letter $S_\LLL$ whatever is the space in which it is 
 considered.   
 It is classical that we may work at the level of the evolution equation and the associated generator or at the level of the associated semigroup.
 We will do the job at both levels. 
 
 As announced, we will  not bother with too much  rigorous justification but rather establish a priori weighted Lebesgue norm estimates from what we may
very classically deduce the well-posedness of the Cauchy problem \eqref{eq:Kolmogorov}--\eqref{eq:KolmoBdyCond}--\eqref{eq:initialDatum} and also deduce the existence of the associated semigroup. 
The solutions of the KFP equations would have to be understood in a appropriate renormalized sense, but again we will not bother about this
 important but technical point and we will freely make the computations  as if the considered  functions are smooth and  fast enough decaying  at infinity. 
 Because the KFP equation conserves the positivity, the associate semigroup is positive and we may thus only handle with nonnegative functions. 
All these issues are discussed in the companion papers \cite{sanchez:hal-04093201,CGMM**,CM-Landau**} for more general classes of KFP equations and we thus refer to these works for more details.

 \smallskip
 We now introduce the class of weight function we deal with. We denote by $\CCC$ the operator 
\beqn\label{def:CCCf}
\CCC f := \Delta_v f + \Div_v(vf), 
\eeqn
which is nothing but the collision part of the Kinetic Fokker-Planck operator involved in \eqref{eq:Kolmogorov}.  
We observe that for $f,\omega : \R^d \to \R_+$ and $p \in [1,\infty)$, we have 
\beqn\label{eq:identCCCffp}
\int_{\R^d} (\CCC f) f^{p-1} \omega^p dv =
 - \frac{4 (p-1)}{ p^2}  \int_{\R^d} |\nabla_v (f\omega)^{p/2} |^2  + \int |f|^p \omega^p \varpi,
\eeqn
with 
\beqn\label{def:varpi}
\varpi = \varpi_{\omega,p}(v) := 2 \left(1 - \frac 1p \right) \frac{|\nabla_v \omega|^2 }{ \omega^2}  + \left(\frac 2p -1\right)\frac{ \Delta_v \omega } \omega  + 
\left(1 - \frac 1p \right) d 
- v  \cdot \frac{\nabla_v \omega}{\omega},
\eeqn
see for instance \cite[Lemma~7.7]{sanchez:hal-04093201} and the references therein. 
We define $\mathfrak W$ as the set of radially symmetric  nondecreasing weight functions  $\omega : \R^d \to (0,\infty)$ such that 
$$
\kappa = \kappa_\omega := \max_{p=1,\infty} 
\sup_{v \in \R^d}  \varpi_{\omega,p} < \infty.
$$

It is worth noticing that  $\omega := \langle v \rangle^k e^{ \zeta \langle v \rangle^s}$, with $k \in \R$ and $s,\zeta \ge 0$, satisfies 
\bean
&&\varpi (v)  \underset{|v|\to \infty}{\sim} (s\zeta)^2 |v|^{2s-2}   -  s \zeta  |v|^{s}  \quad \hbox{if} \quad s > 0, 
\\
&&\varpi (v)  \underset{|v|\to \infty}{\sim}  \frac{d}{p'}-k  \quad \hbox{if} \quad s = 0, 
\eean
so that $\omega \in \mathfrak W$ when  
\beqn\label{eq:condition-example-in-W}
s \in (0,2), \quad\text{or}\quad s =2 \;\text{ and }\; \zeta < 1/2 ,\quad\text{or}\quad s=0.  
\eeqn
On the other hand, we may check 
\beqn\label{eq:calcul-varpiMMM}
\varpi_{\MMM^{-1+1/q},p} (v) = - \frac{1}{q} \left( 1-\frac{1}{q}\right) |v|^2 + \left( \frac{1}{p}+\frac{1}{q} - \frac{2}{pq}\right) d, 
\eeqn
so that  
for the limit case $\omega = \MMM^{-1} \in \mathfrak W$, since then 
$\varpi_{\MMM^{-1},p} \equiv 2d/p$. We finally define 
\beqn\label{def:frakW0} 
\mathfrak W_0 := \left\{ \omega \in \mathfrak W \,; \,  1 \lesssim \omega \lesssim \MMM^{-1}, \ \omega^{-1} |v|, \, \omega \MMM |v| \in L^1(\R^d) \right\}.
\eeqn

\begin{prop}\label{prop:EstimLp}
For any weight function $\omega \in \mathfrak W_0$, there exist $\kappa \ge 0$ and $C \ge 1$ such that for  any exponent $p \in [1,\infty]$ and any solution $f$ to the KFP equation 
\eqref{eq:Kolmogorov}--\eqref{eq:KolmoBdyCond}--\eqref{eq:initialDatum},  there holds  
\beqn\label{eq:prop:EstimLp}
\| f_t \|_{L^p_\omega} \le C e^{\kappa t} \| f_0 \|_{L^p_\omega} , \quad \forall \, t \ge 0,
\eeqn
and we write equivalently 
\beqn\label{eq:prop:SGEstimLp}
S_\LLL(t) : L^p_\omega \to L^p_\omega, \ \hbox{ with growth rate } \ \OO(e^{\kappa t}), \quad \forall \, t \ge 0. 
\eeqn
\end{prop}

\smallskip
We start recalling the following classical estimate based on very specific choices of the weight functions, so that Darroz\`es-Guiraud type inequality~\cite{DGineq} may be used.

\begin{lem}\label{lem:B11:lem0} 
For any $p \in [1,\infty]$, the semigroup $S_{\LLL}$ is a contraction on $L^p_{\MMM^{-1+1/p}}$. 
\end{lem}

\begin{proof}[Proof of Lemma~\ref{lem:B11:lem0}] We fix  $p \in [1,\infty)$, $0 \le f_0 \in L^p_{\MMM^{-1+1/p}}$ and 
we denote by   $f =f (t,x,v) \ge 0$ the  solution to the Cauchy problem associated to \eqref{eq:Kolmogorov}--\eqref{eq:KolmoBdyCond}--\eqref{eq:initialDatum}.
We compute
\bean
\frac{1}{p}\frac{d}{dt} \int_{\OO} f^p \MMM^{1-p} 
&=&   \int_{\OO}   (\CCC f)  f^{p-1} \MMM^{1-p} - \frac1p\int_{\Sigma} (\gamma f)^p \MMM^{1-p} \, n_x \cdot v 
\\
&\le&  \int_{\OO}  \varpi_{\MMM^{-1+1/p},p} f^p \MMM^{1-p} 
- \frac{1}{p} \int_{\Sigma_{+}} (\gamma_+ f)^p \MMM^{1-p}  |n_x \cdot v| \\
&&
+ \frac{1}{p}\int_{\Sigma_{-}} \left\{ (1-\iota) \SSS \gamma_+ f + \iota \DDD \gamma_+ f \right\}^p \MMM^{1-p}  |n_x \cdot v|, 
\eean
where we have used the Green-Ostrogradski formula in the first line, we have thrown away the first term coming from \eqref{eq:identCCCffp} in the second line, we have split the boundary term into two pieces and we have used the boundary condition on its incoming part in the second and third lines. For the last term we have 
\bean
&& \int_{\Sigma_{-}} \left\{ (1-\iota) \SSS \gamma_+ f + \iota \DDD \gamma_+ f \right\}^p \MMM^{1-p}  |n_x \cdot v|
\\
&&\le  \int_{\Sigma_{-}}   (1-\iota) (\SSS \gamma_+ f)^p \MMM^{1-p}  |n_x \cdot v| +\int_{\Sigma_{-}}   \iota  (\widetilde{\gamma_+ f})^p   \MMM   |n_x \cdot v|
\\
&&\le  \int_{\Sigma_{+}}   (1-\iota) (\gamma_+ f)^p \MMM^{1-p}  |n_x \cdot v| + \int_{\partial\Omega}   \iota   (\widetilde{\gamma_+ f})^p, 
\eean
where we have used the convexity of the function $s \mapsto s^p$ in the second line and we have used both the change of variables $v \mapsto \VV_x v$ in the last integral (which transforms $\Sigma_-$ into $\Sigma_+$ with unit Jacobian) and the normalization condition on $\MMM$ (see \eqref{eq:FPK-def_M}) in the third line. 
Observing next that 
\bean
 (\widetilde{\gamma_+ f})^p 
 &=& \Bigl(\int_{\Sigma^x_+} (\gamma_+ f/\MMM) \, \MMM  |n_x \cdot v| \, dv \Bigr)^p
\\
&\le& \int_{\Sigma^x_+} (\gamma_+ f/\MMM)^p \, \MMM  |n_x \cdot v| \, dv, 
\eean
thanks to the Jensen inequality (also called Darroz\`es-Guiraud's inequality in this context!), which is true  because of  the normalization condition on $\MMM$. We have thus established
\bean
 \int_{\Sigma_{-}} \left\{ (1-\iota) \SSS \gamma_+ f + \iota \DDD \gamma_+ f \right\}^p \MMM^{1-p}  |n_x \cdot v|
\le  \int_{\Sigma_{+}}    (\gamma_+ f)^p \MMM^{1-p}  |n_x \cdot v|,
\eean
from which we obtain 
\bean
\frac{d}{dt} \int_{\OO} f^p \MMM^{1-p} 
\le p \int_{\OO}  \varpi_{\MMM^{-1+1/p},p} f^p \MMM^{1-p}.
\eean
Coming back to \eqref{eq:calcul-varpiMMM}, we observe that 
$$
\varpi_{\MMM^{-1+1/p},p}
(v) 
= - \frac{1}{p} \left( 1-\frac{1}{p}\right) |v|^2 + \frac{2}{p}\left( 1-\frac{1}{p}\right) d \le 0,
$$
from what we immediately deduce that  $S_\LLL$ is a contraction on $L^p_{\MMM^{-1+1/p}}$ when $p \in [1,\infty)$. 
We get the same conclusion in $L^\infty_{\MMM^{-1}}$ by letting $p \to\infty$.
\end{proof}

\smallskip
We extend the decay estimate to a general weight function in a $L^1$ framework by using an appropriate  modification of the initial weight. 
That kind of moment estimate is reminiscent of $L^1$ hypodissipativity techniques, see e.g.\ \cite{MR3489637,MR3779780,MR4179249}.
Our multiplicator is inspired from the usual multiplicator used in order to control the diffusive operator in previous works on the Boltzmann equation, see e.g.\ \cite{ArkerydMaslova,MR1776840,MR2721875,MR4179249}. 
For further references, we define the formal adjoints
\beqn\label{def:CCC*}
\LLL^* := v \cdot \nabla_x + \CCC^*, \quad \CCC^* g:=   \Delta_v g - v \cdot \nabla_v g.
\eeqn

\begin{lem}\label{lem:SLestimL1} 
 Let $\omega : \R^d \to (0,\infty)$ be a radially symmetric nondecreasing
weight function such that $\omega \in \mathfrak W$ and $\MMM \omega |v| \in L^1(\R^d)$. There exists $\kappa \ge 0$ such that we have
$$
S_{\LLL}(t) : L^1(\omega) \to L^1(\omega), \quad \forall \, t \ge 0, 
$$ 
with growth estimate $\OO(e^{\kappa t})$. 
\end{lem}

It is worth emphasizing that with a very similar proof we may establish the same growth rate in $L^p_\omega$ for $p \in (1,\infty)$, but we were not able to reach the limit exponent $p = \infty$ because our estimates blow up as $p\to\infty$.

\begin{proof}[Proof of Lemma \ref{lem:SLestimL1}] 
  Without loss of generality we may suppose that $\omega \ge 1$. 
We split the proof into two steps. 

\smallskip
\noindent{\sl Step 1. }
For $ 0 \le f_0 \in L^1(\omega)$, 
we denote by   $f = f(t,x,v) \ge 0$ the  solution to the Cauchy problem \eqref{eq:Kolmogorov}--\eqref{eq:KolmoBdyCond}--\eqref{eq:initialDatum}, 
so that $f(t) = S_{\LLL}(t) f_0$.

We introduce the weight functions 
$$
\omega_A(v) := \chi_A(v) + (1-\chi_A(v)) \omega (v), 
$$
with $\chi_A(v) := \chi(|v|/A)$, $A \ge 1$ to be chosen later and $\chi \in C^2(\R_+)$, ${\bf 1}_{[0,1]} \le \chi \le {\bf 1}_{[0,2]}$, and next 
$$
\widetilde \omega (x,v)  :=  \omega_A(v) + \frac12 n_x \cdot \widetilde v,  
$$
with $\hat v := v/\langle v \rangle$ and $\widetilde{v} :=  \hat v/\langle v \rangle $. It is worth emphasizing that 
\beqn\label{eq:omega&omegatilde}
1 \le \omega_A \le   \omega  \quad \hbox{and}\quad
c_A^{-1} \omega \le \tfrac12\omega_A \le \widetilde \omega  \le \tfrac32 \omega_A, 
\eeqn
with $c_A  \in (0,\infty)$. 
We write 
\beqn\label{eq1:SLestimL1}
\frac{d}{dt} \int_{\OO} f \, \widetilde \omega 
= \int_{\OO} f \, \LLL^* \widetilde \omega  
- \int_\Sigma \gamma f \, \widetilde \omega  \, n_x \cdot v.
\eeqn
We first compute separately each contribution of the boundary term
\bean
B := - \int_\Sigma \gamma f \, \widetilde \omega \,  n_x \cdot v = B_1 + B_2,
\eean
with
\bean
B_1 &:=& - \int_{\Sigma_{+}} \gamma_+ f \omega_A  |n_x \cdot v| 
+ \int_{\Sigma_{-}} \left\{ (1-\iota) \SSS \gamma_+ f + \iota \DDD \gamma_{+} f \right\} \omega_A |n_x \cdot v| 
\\
B_2 &:=& -   \frac12 \int_\Sigma \gamma  f \,   (n_x \cdot \hat v)^2. 
\eean
Making the change of variables $  v \mapsto \VV_x v$ in the last integral involved in $B_1$, 
we get
$$
\begin{aligned}
B_1 =- \int_{\Sigma_{+}} \iota \gamma_+ f \, \omega_A \,  |n_x \cdot v| 
+ \int_{\Sigma_{+}} \iota \DDD \gamma_{+} f \, \omega_A \,  |n_x \cdot v| .
\end{aligned}
$$
We then define 
\begin{equation}\label{eq:Komega}
K_1(\omega_A ) := \int_{\R^d}  \MMM \, \omega_A  \, (n_x \cdot v)_{+} \, dv,
\end{equation}
which is finite by the assumption on $\omega$, so that 
$$
\int_{\Sigma_{+}} \iota \DDD \gamma_{+} f \, \omega_A \,  |n_x \cdot v| =  \int_{\partial\Omega} \iota K_1(\omega_A) \widetilde{\gamma_{+} f} .
$$
Since $\omega_A \ge 1$, we then obtain  
$$
B_1 \le  \int_{\partial\Omega} \iota (K_1(\omega_A) - 1) \widetilde{\gamma_{+} f} .
$$
On the other hand, denoting
\begin{equation}\label{eq:kappa}
K_0 := \int_{\R^d} \MMM (n_x \cdot \hat v)_+^2  \, dv \in (0,\infty), 
\end{equation}
which we observe is independent of $x$, we have 
\bean
- \int_{\Sigma} \gamma f  (n_x \cdot \hat v)^2
&\le& - \int_{\Sigma_{+}} \iota \DDD \gamma_{+} f   (n_x \cdot \hat v)^2  
=
-  K_0 \int_{\partial\Omega}  \iota \widetilde{\gamma_{+} f} .
\eean
Recalling \eqref{eq:omega&omegatilde} and observing that 
 $\omega_A \to 1$ a.e.\ when $A \to \infty$, we get  $K_1(\omega_A) \to K_1(1) = 1$ 
as $A \to \infty$ thanks to the dominated convergence Theorem of Lebesgue and the normalization condition on $\MMM$. We may thus fix $A \ge 1$ large enough in such a way that
$$
K_1(\omega_A) - 1 - \frac12 K_0 \le 0,
$$
and the contribution of the boundary is nonpositive. 

\smallskip
\noindent{\sl Step 2. }
For the contribution of the volume integral, we write 
\bean
 \LLL^* \widetilde \omega
= \CCC^* \omega +  \CCC^* [\chi_A(1-\omega)] +   \CCC^* [n_x \cdot \widetilde v] + v \cdot \nabla_x (n_x \cdot \widetilde v), 
\eean
  where we recall that  the adjoint   Fokker-Planck operator $\CCC^*$ is defined in \eqref{def:CCC*}. 
Because $\omega \in \mathfrak W$, we have 
$$
\CCC^* \omega \le \varpi_{\omega,1} \, \omega \le \kappa_1 \omega, 
$$
for some $\kappa_1 \in \R$. On the other hand, because $\chi_A$ has compact support and because of the regularity assumption of $\Omega$, we have 
$$
 \CCC^* [\chi_A(1-\omega)] +   \CCC^* [n_x \cdot \widetilde v] + v \cdot \nabla_x (n_x \cdot \widetilde v) \le \kappa_2, 
 $$
 for some $\kappa_2 \in \R_+$. Coming back to \eqref{eq1:SLestimL1}, we deduce that 
\bean
\frac{d}{dt} \int_{\OO} f \, \widetilde \omega \le \kappa  \int_{\OO} f \,  \widetilde \omega,
\eean
with $\kappa := 2 \kappa_1 + c_A\kappa_2$. We immediately conclude thanks to Gr\"onwall's lemma and the comparison \eqref{eq:omega&omegatilde}  between $\omega$ and $\widetilde\omega$. 
\end{proof}

\medskip
We establish now a similar exponential growth estimate in a  general weighted $L^1$ framework for the dual  backward problem associated to \eqref{eq:Kolmogorov}--\eqref{eq:KolmoBdyCond}--\eqref{eq:initialDatum}, namely
\begin{equation}\label{eq:dualKFP}
\left\{
\bal
- \partial_t g &=  v \cdot \nabla_x g + \CCC^* g    \quad&\text{in} \quad (0,T) \times \OO ,\\
\gamma_+ g &=  \RRR^* \gamma_- g   \quad &\text{on} \quad (0,T) \times \Sigma_+ , \\ 
 g(T) &= g_T \quad &\text{in} \quad \OO,  
\eal
\right.
\end{equation}
for any $T \in (0,\infty)$ and any final datum $g_T$.
The adjoint Fokker-Planck operator $\CCC^*$ is defined in \eqref{def:CCC*}, and the adjoint reflection operator $\RRR^*$ is defined by 
$$
\RRR^* g (x,v) = (1-\iota)  \SSS g (x,v) + \iota \DDD^* g (x), $$
with
$$ 
\DDD^* g (x) = \widetilde{\MMM g}(x) := \int_{\R^d} g(x,w) \MMM(w) (n_x \cdot w)_{-} \, dw.
$$ 
Again, we do not discuss the very classical issue about well-posedness in Lebesgue spaces for these problems nor the possibility to approximate the solutions by {\it smooth enough solutions}, which is useful in the following argument. Consider $f$ a solution to the forward Cauchy problem \eqref{eq:Kolmogorov}--\eqref{eq:KolmoBdyCond}--\eqref{eq:initialDatum} and $g$ a solution to the above dual problem  \eqref{eq:dualKFP}. 
We compute (at least formally)
\bean
\int_\OO f(T) g_T  
&=& \int_\OO f_0 g(0) + \int_0^T \int_\OO (\partial_t f g + f \partial_t g  )  \, ds \\
&=& \int_\OO f_0 g(0) - \int_0^T \!\!\! \int_\OO (v \cdot \nabla_x f  g + f v \cdot \nabla_x g ) \, ds
\\
&=& \int_\OO f_0 g(0) - \int_0^T \!\!\! \int_\Sigma (v \cdot n) \gamma f \gamma g \, ds
\\
&=& \int_\OO f_0 g(0) - \int_0^T \!\!\! \int_{\Sigma_+} (v \cdot n) (\gamma_+ f)  (\RRR^* \gamma_- g) \, ds
\\
&& \quad +  \int_0^T \!\!\! \int_{\Sigma_-} |v \cdot n| (\RRR \gamma_+ f) ( \gamma_- g)   \, ds,
\eean 
by using the Green-Ostrogradski formula and the reflection conditions at the boundary. 
From the very definition of $\RRR$ and $\RRR^*$, we then deduce  the usual identity
\begin{equation}\label{eq:identite-dualite}
\int_\OO f(T) g_T   
= \int_\OO f_0 g(0), 
\end{equation}
or equivalently that $g(t) = S^*_\LLL(T-t) g_T$. 
We observe now that for a weight function $\omega$, we have 
\beqn\label{eq:CComega-1}
 \CCC \omega^{-1} = \omega^{-1} \varpi_{\omega,\infty}.  
\eeqn
We then define  $\mathfrak N$ the class of weight functions $m : \R^d \to (0,\infty)$ such that 
 $\omega = m^{-1} \in \mathfrak W$. In particular, 
because of  \eqref{eq:CComega-1} and the definition of $ \mathfrak W$, there exists $\kappa' \in \R$ such that 
\beqn\label{eq:CCm}
\CCC m \le \kappa' m .  
\eeqn
We also define 
\beqn\label{def:frakM0}
\mathfrak N_0 := \{ m \in \mathfrak N; \ \MMM \lesssim m, \ m v \in L^1(\R^d) \}.
\eeqn
For further discussion, we emphasize that $\omega \in \mathfrak W_0$ clearly implies $\omega^{-1} \in \mathfrak N_0$.

\begin{lem}\label{lem:SLestimLinfty}
For any weight function $m \in \mathfrak N_0$, 
there exists $\kappa \in \R$ such that 
$$
S^*_\LLL(t) : L^1_m \to L^1_m, \quad \OO(e^{\kappa t}).
$$ 
More precisely,  there exists $C \ge 1$ such that for any $T > 0$ and any $g_T \in L^1_m$, the associated solution $g$ to the backward dual problem \eqref{eq:dualKFP} satisfies 
\begin{equation}\label{eq:g0&g-<gtC}
\| g(0) \|_{L^1_m}  
\le C e^{\kappa T} \| g_T \|_{L^1_m}. 
\end{equation}
\end{lem}

\begin{proof}[Proof of Lemma~\ref{lem:SLestimLinfty}]    
Without loss of generality we may suppose that $m \ge \MMM$. 
For $T \in (0,\infty)$ and $ 0 \le g_T \in L^1_m$,  we denote by   $g = g(t,x,v)$ the  solution to the backward dual Cauchy problem \eqref{eq:dualKFP}. 
We introduce the weight functions 
\beqn\label{def:mA}
m_A := \chi_A \MMM + (1-\chi_A)m, 
\quad
\widetilde m   :=  m_A - \frac12 (n_x \cdot \widetilde v) \, \MMM,  
\eeqn
with the notations of  Lemma \ref{lem:SLestimL1}.  It is worth emphasizing that 
\beqn\label{eq:m&mtilde} 
\MMM \le m_A \le m \quad\hbox{and}\quad
c_A^{-1} m \le \tfrac12m_A \le \widetilde m  \le \tfrac32 m_A, 
\eeqn
with $c_A \in (0,\infty)$. Similarly as in the proof of Lemma \ref{lem:SLestimL1}, we compute
\bean
- \frac{d}{dt} \int_\OO g \, m_A 
&=& \int_{\OO}  g  \, (\CCC m_A)   +  \int_\Sigma \gamma g  \, m_A \, n_x \cdot v 
\\
&=& \int_{\OO} g   \, (\CCC m_A)  + \int_{\Sigma_+} [ (1-\iota) \SSS \gamma_- g  + \iota \widetilde{  \gamma_- g \MMM}]  \, m_A \, |n_x \cdot v| 
-  \int_{\Sigma_-}  \gamma_- g  \, m_A \, |n_x \cdot v| , 
\eean
where we have used again the Green-Ostrogradski formula in the first line and the reflection condition at the boundary in the second line.
We deduce 
\bean
- \frac{d}{dt} \int_\OO g \, m_A 
&=& \int_\OO  g  \, (\CCC m_A)  -  \int_{\Sigma_-} \iota  \gamma_- g \, m_A \, |n_x \cdot v| 
+ \left( \int_{\R^d} m_A (n_x \cdot v)_+ \, dv \right) \int_{\Sigma_-} \iota  \MMM \gamma_- g  \, |n_x \cdot v|, 
\eean
by making the change of variables $v \mapsto \VV_x v$ on the {\it outgoing part} $\Sigma_+$ of the boundary (which is in fact the {\it incoming part} of the boundary for the backward dual problem). 
Since $m_A \ge \MMM$, we have established a first estimate 
\bean
-\frac{d}{dt} \int_\OO g \, m_A 
&\le&  \int_{\OO}  g  \, (\CCC m_A)   +   \int_{\Sigma_-} \iota (K_1(m_A) - 1)   \MMM \gamma_- g \, |n_x \cdot v|,
\eean
with now
$$
K_1(m_A) :=  \int_{\R^d}  m_A \, (n_x \cdot   v)_+ \, dv    \to 1, \quad \hbox{as} \quad A \to \infty.
$$
On the other hand, with the same notations as in the proof of  Lemma~\ref{lem:SLestimL1}, we have 
\bean
\frac{d}{dt} \int_\OO g \MMM \,   (n_x \cdot   \widetilde{v}) 
&=& \int_{\OO}   g  \CCC( \MMM \,   (n_x \cdot   \widetilde{v}) ) 
- \int_{\OO}   g  \,  \MMM  \, \hat v \cdot D_x n_x \hat v  
-  \int_\Sigma \gamma g \, \MMM \, (n_x \cdot \hat v)^2.
\eean
For the last term, there holds
\bean
\int_\Sigma \gamma g \MMM  \, (n_x \cdot \hat v)^2  
&\ge&   \int_{\Sigma_+} \iota (\DDD^* \gamma_- g) (n_x \cdot \hat v)^2   \, \MMM
\\
&\ge&  \left( \int_{\R^d} \MMM (n_x \cdot \hat v)_+^2      \right)   \int_{\Sigma_-} \iota \MMM \gamma_- g |n_x \cdot v| ,
\eean
which implies a second estimate 
\bean
\frac{d}{dt} \int_\OO g \MMM \, (n_x \cdot  \widetilde{v} )
&\le&  \int_\OO   g  \CCC( \MMM \,   (n_x \cdot   \widetilde{v}) ) 
- \int_\OO   g  \, \MMM  \, \hat v \cdot D_x n_x \hat v  
-  K_0 \int_{\Sigma_-} \iota \MMM \gamma_- g |n_x \cdot v|,
\eean
with now
\begin{equation}\label{eq:kappa0}
K_0 := \int_{\R^d}  \MMM \, (n_x \cdot \hat v)_+^2  \, dv   \in (0,\infty).
\end{equation}
Choosing $A > 0$ large enough such that $K_1(m_A) - 1 - \tfrac12 K_0 \le 0$, the contribution of the boundary is nonpositive and we obtain 
\bean
-\frac{d}{dt} \int_\OO g \, \widetilde m 
&\le&  \int_{\OO}  g  \, [ \CCC m +  \CCC [\chi_A(\MMM - m)] +   \CCC [n_x \cdot \widetilde v \MMM] - v \cdot \nabla_x (n_x \cdot \widetilde v \MMM) ]
\\
&\le& \kappa \int_{\OO}  g   \widetilde m,
\eean
for some $\kappa \in \R$, by arguing similarly as during the proof of Lemma \ref{lem:SLestimL1} and in particular using \eqref{eq:CCm}. By the Gr\"onwall's lemma, we then deduce 
\begin{equation}\label{g0&g-<gt}
\| g(0) \|_{L^1(\tilde m)}  
\le e^{\kappa T} \| g_T \|_{L^1(\tilde m)},
\end{equation}
from which we immediately  conclude to \eqref{eq:g0&g-<gtC}. \end{proof}

We may now come to the proof of the main result of this section. 

\begin{proof}[Proof of Proposition~\ref{prop:EstimLp}] 
For $f_0 \in L^\infty_\omega$, let us define $f (t):= S_\LLL(t) f_0$ the associated flow. 
Because of the duality identity \eqref{eq:identite-dualite}, for any $g_t \in L^1_{\omega^{-1}}$, we have 
\bean
\int_\OO f(t) g_t   
= \int_\OO f_0 g(0)
\le \| f_0 \|_{L^\infty_\omega} \| g(0) \|_{L^1_{\omega^{-1}} }.
\eean 
Together with \eqref{eq:g0&g-<gtC}, we deduce 
\bean
\int_\OO f(t) g_t   \le \| f_0 \|_{L^\infty_\omega} C e^{\kappa t} \| g_t \|_{L^1_{\omega^{-1}}}. 
 \eean 
 Taking the supremum on $g_t$ over the unit ball of $L^1(\omega^{-1})$, we thus conclude that 
$$
\| f(t) \|_{L^\infty_\omega} \le C e^{\kappa t} \| f_0 \|_{L^\infty_\omega}, 
$$
for any $f_0$, which is the desired estimate \eqref{eq:prop:SGEstimLp} when $p=\infty$. The estimate \eqref{eq:prop:SGEstimLp} for $p=1$ has been established in Lemma~\ref{lem:SLestimLinfty}.
We then conclude to the estimate  \eqref{eq:prop:SGEstimLp} for any $p \in [1,\infty]$ by using a standard interpolation argument. 
\end{proof}

\begin{rem}
The conditions on the weight function $\omega$ in the statement of Proposition~\ref{prop:EstimLp} are not optimal but they are more than enough for our purpose. As a matter of fact, we may observe that 

- Lemma~\ref{lem:B11:lem0} gives an estimate on $S_\LLL$ in  $L^1$ and in $L^\infty_{\MMM^{-1}}$; 

- Lemma \ref{lem:SLestimL1} gives an estimate on $S_\LLL$ in  $L^1_\omega$ from $\omega = 1$ and up to $\omega =  \MMM^{-1} \langle v \rangle^{-d-1-\eps}$, $\eps > 0$; 

- Lemma \ref{lem:SLestimLinfty} gives an estimate  on $S_\LLL^*$ in $L^1_m$  from $m = \langle v \rangle^{-d-1-\eps}$, $\eps > 0$,  and up to $m =  \MMM$, and thus an estimate on $S_\LLL$ in  $L^\infty_\omega$ from 
$\omega =    \langle v \rangle^{d+1+\eps}$, $\eps > 0$,  and up to $\omega =  \MMM^{-1}$.

- We may straightforwardly check that $\omega := \langle v \rangle^k e^{\zeta \langle v \rangle^s} \in \frak W_0$ when 
\beqn\label{eq:condition-example-in-W0}
s \in (0,2), \quad\text{or}\quad s =2 \;\text{ and }\; \zeta < 1/2 , \quad\text{or}\quad s=0    \;\text{ and }\;  k > d+1.
\eeqn
\end{rem}

\section{Ultracontractivity: Proof of Theorem~\ref{th:DGNML1Linfty}}
 \label{sec:proofTh1}

 \subsection{An improved weighted $L^2$ estimate at the boundary}

 The DeGiorgi-Nash-Moser theory tells us that for parabolic equations some gain of integrability estimates can be obtained by elementary manipulations  when evaluating the evolution of functions $f^q$ for $q \not=1$. That  kind of regularity effect is also called ultracontractivity. 
  More recently, a similar theory has been developed for the Kolmogorov equation in the whole space, see in particular  \cite{MR2068847,MR3923847}.
  Our purpose is   to generalize these techniques to a bounded domain framework. 
  In the present framework and in order to be able to deduce next (by interpolation) the same kind of regularity effect in the border $L^1_\omega$ space, we   first consider $q < 1$. 
 Let us observe that for $q \not= 0$ and $f$ a positive solution to the  KFP equation \eqref{eq:Kolmogorov}, we may compute 
$$
 \partial_t f^q +   v \cdot\nabla_x f^q  - v \cdot\nabla_v f^q -qd f^q  -  \Delta_v f^q - 4 \frac{(1-q)}q |\nabla_vf^{q/2}|^2 = 0.
 $$
Multiplying the equation by $\Phi^q := \varphi^q m^q$ with $q \in (0,1)$, $\varphi \in \DD((0,T))$, and integrating in all the variables, we obtain
\beqn\label{eq:weakKolmogorovLq}
\frac1q\int_\Gamma (\gamma f)^q \Phi^q  n_x \cdot v + \frac1q \int_\UU f^q \TT^* \Phi^q
= 4 \frac{(1-q)}{q^2}   \int_\UU |\nabla_v (f\Phi)^{q/2}|^2 + \int_\UU f^q \Phi^q \varpi,
\eeqn
with  $\varpi := \varpi_{m,q}$ defined in  \eqref{def:varpi} and 
\beqn\label{def:TT*}
\TT^* \Psi := - \partial_t \Psi - v \cdot \nabla_x \Psi. 
\eeqn
Alternatively, defining 
\beqn\label{def:TT}
\TT := \partial_t + v \cdot \nabla_x
\eeqn
and recalling that $\CCC$ has been defined in \eqref{def:CCCf}, we may write 
$$
\TT \frac{f^q }{ q} = f^{q-1} \TT f = f^{q-1} \CCC f,
$$
so that 
\bean
\frac1q\int_\Gamma (\gamma f)^q \Phi^q  n_x \cdot v + \frac1q \int_\UU f^q \TT^* \Phi^q
= \int_\UU f^{q-1} (\CCC f)   \Phi^q, 
\eean
from what we deduce \eqref{eq:weakKolmogorovLq} with the help of \eqref{eq:identCCCffp}-\eqref{def:varpi}.

 \smallskip
 We now establish a key new moment estimate on  the KFP equation  \eqref{eq:Kolmogorov}--\eqref{eq:KolmoBdyCond}--\eqref{eq:initialDatum} which makes possible to control a solution near the boundary. 
 The proof is based on the introduction of an appropriate weight function which combines the twisting term used in the previous section and the  twisting term used in \cite[Section~11]{sanchez:hal-04093201}, that last one being in the spirit of  moment arguments used in \cite{MR1166050,MR3591133}.

\begin{prop}\label{prop:TheLqEstim}
Let $q \in (0,1)$ and $m : \R^d \to (0,\infty)$ be a radially symmetric decreasing  weight function such that   $m^{\frac{q}{1-q}} |v| \in L^1(\R^d)$.
There exists  $C = C(q,m,\Omega)>0$  such that for any nonnegative solution $f$  to the KFP equation \eqref{eq:Kolmogorov}--\eqref{eq:KolmoBdyCond}--\eqref{eq:initialDatum} and any test function $0 \le \varphi \in \DD((0,T))$,  there holds  
$$
\int_\UU f^q \widetilde m^q  \frac{(n_x \cdot \hat v)^{2}}{\delta^{1/2}} \varphi^q+  
\int_\UU  |\nabla_v(f^{q/2} \widetilde m^{q/2})|^2  \varphi^q
\le 
C \int_\UU  f^q m^q [|\partial_t \varphi^q| +  \langle \varpi_- \rangle\varphi^q],
$$
where  $\widetilde m$ is a   modified weight function such that $ m \lesssim \widetilde m \lesssim  m$ and $\varpi := \varpi_{\tilde m,q}$ is defined in  \eqref{def:varpi}.

\end{prop}

\begin{proof}[Proof of Proposition~\ref{prop:TheLqEstim}] 
We fix $q \in (0,1)$ and we introduce the modified weight functions 
\beqn\label{def:mABIS}
m_A^q := \chi_A \MMM^{1-q} + (1-\chi_A) m^q, 
\eeqn
for $A \ge 1$ and with the notations of  Lemma \ref{lem:SLestimL1}. We next introduce the function
\bean
\Phi^q := \varphi^q \,  \widetilde m^q, \quad   \widetilde m^q :=   m_A^q - \frac{ m_A^q  }{ 4}  n_x \cdot \widetilde v + \frac{m_A^q }{ 4 D^{1/2}} \delta(x)^{1/2}  n_x \cdot \widetilde v ,
\eean
where  $D= \sup\delta$ is half the diameter of $\Omega$, so that in particular an estimate similar to \eqref{eq:m&mtilde} holds. From \eqref{eq:weakKolmogorovLq}, we have 
\begin{equation}\label{eq:integrate_fq}
\begin{aligned}
& 4\frac{(1-q)}{q}   \int_\UU |\nabla_v (f\Phi)^{q/2}|^2 - \int_\Gamma (\gamma f)^q \Phi^q  n_x \cdot v - \int_\UU f^q \TT^*_2 \Psi_3
 \\
&= \int_\UU f^q \TT^*_2 \Psi_{12} - q  \int_\UU f^q \Phi^q \varpi + \int_\UU f^q \TT^*_1 \Phi^q, 
\end{aligned}
\end{equation}
where $\TT^*_1 = - \partial_t$, $\TT^*_2 = - v \cdot \nabla_x$, $\varpi = \varpi_{\tilde m,q}$  and 
\bean
 \Psi_{12} := \varphi^q m_A^q \left  ( 1  -   \frac{ 1 }{ 4} n_x \cdot \widetilde v \right), \quad
\Psi_3 :=  \varphi^q \, \frac {m_A^q }{ 4 D^{1/2}} \delta(x)^{1/2}  n_x \cdot \widetilde v.
\eean

\medskip
We now compute each term separately.

\smallskip\noindent
\textit{Step 1.}
For the second term at the left-hand side of \eqref{eq:integrate_fq}, we observe that 
$$
- \int_\Sigma (\gamma f)^q m_A^q n_x \cdot v
= - \int_{\Sigma_+} (\gamma_{+} f)^q m_A^q |n_x \cdot v| +  \int_{\Sigma_-} (\gamma_{-} f)^q m_A^q |n_x \cdot v|
$$
and, using the boundary condition together with the fact that the map $s \mapsto s^q$ is concave, we get
$$
\begin{aligned}
&\int_{\Sigma_-} \left\{ (1-\iota)\SSS \gamma_{+} f + \iota \DDD \gamma_{+} f \right\}^q m_A^q |n_x \cdot v| \\
&\qquad \ge \int_{\Sigma_-}  (1-\iota)(\SSS \gamma_{+} f)^q m_A^q |n_x \cdot v|
+ \int_{\Sigma_-} \iota ( \widetilde{\gamma_{+} f})^q \MMM^q m_A^q |n_x \cdot v|.
\end{aligned}
$$
Removing the contribution of the specular reflection thanks to the change of variables $v \mapsto \VV_x v$ as in the proof of Lemmas~\ref{lem:SLestimL1} and~\ref{lem:SLestimLinfty}
and using the H\"older inequality in order to manage the term involving $K_2$, we  therefore obtain
\bean
- \int_\Sigma (\gamma f)^q m_A^q n_x \cdot v
&\ge& 
\int_{\Sigma_-} \iota (\widetilde{\gamma_+ f})^q \MMM^q m_A^q (n_x \cdot v)_- - \int_{\Sigma_+} \iota (\gamma_+ f)^q m_A^q (n_x \cdot v)_+
\\
&\ge& (K_1(m_A) - K_2(m_A)^{1-q}) \int_{\partial\Omega}  \iota (\widetilde{\gamma_+ f})^q,
\eean
with 
$$
K_1 (m_A) := \int_{\R^d}  \MMM^q m_A^q (n_x \cdot v)_- dv < + \infty, \quad K_2 (m_A) := \int_{\R^d}   m_A^{\tfrac q {1-q}}  (n_x \cdot v)_+ dv < + \infty. 
$$
On the other hand, we have 
\bean
 \int_\Sigma (\gamma f)^q m_A^q \frac{(n_x \cdot \hat v)^2}4 
&\ge& K_0(m_A)   \int_{\partial\Omega}  \iota (\widetilde{\gamma_+ f})^q
\eean
with 
$$
K_0 (m_A) := \frac14  \int_{\R^d}  \MMM^q m_A^q (n_x \cdot \hat v)^2_- dv.
$$
Both together, we obtain 
$$
- \int_\Sigma (\gamma f)^q \widetilde m^q  n_x \cdot v \ge \left[ K_0(m_A) + K_1(m_A) - K_2(m_A)^{1-q} \right] \int_{\partial\Omega}  \iota (\widetilde{\gamma_+ f})^q.
$$
Observing that $m_A \to \MMM^{\frac{1}{q}-1}$ when $A \to \infty$, we deduce that $K_1(m_A) \to K_1(\MMM^{\frac{1}{q}-1}) = 1$, $K_2(m_A) \to K_2(\MMM^{\frac{1}{q}-1}) = 1$ and $K_0(m_A) \to K_0(\MMM^{\frac{1}{q}-1}) > 0$
 as $A \to \infty$, thanks to the integrability condition made on $m$ and the dominated convergence theorem of Lebesgue. We may thus choose $A > 0$ large enough in such a way that  
\beqn\label{eq:conditionK1K2K3}
K_0(m_A) + K_1(m_A) - K_2(m_A)^{1-q} \ge 0.
\eeqn

\smallskip\noindent
\textit{Step 2.}
In order to deal with the third term at the left-hand side of \eqref{eq:integrate_fq}, we define  $\psi := \delta(x)^{1/2}  n_x \cdot \widetilde v$. Observing that 
 $\langle v \rangle \psi \in L^\infty(\OO)$, $ \nabla_v \psi \in L^\infty(\OO)$ and 
$$
v \cdot \nabla_x \psi = \frac12 \frac1{\delta(x)^{1/2}} (\hat v \cdot n_x)^2 +    \delta(x)^{1/2} \hat v \cdot D_x n_x \hat v, 
$$
we compute
\bean
- \int_\UU f^q \TT^*_2 \Psi_3  =
\frac{1 }{ 4 D^{1/2}}  \int_\UU f^q \varphi^q m_A^q \left\{ \frac12 \frac1{\delta(x)^{1/2}} (\hat v \cdot n_x)^2 +    \delta(x)^{1/2} \hat v \cdot D_x n_x \hat v \right\}.
\eean
We may now conclude. 
Because of \eqref{eq:conditionK1K2K3}, we may get rid of the boundary term, and together with the last inequality, we get  
\bean
 &&  4\frac{1-q }{ q}   \int_\UU |\nabla_v (f \widetilde m)^{q/2}|^2 \varphi^q  + \frac{1 }{ 8 D^{1/2}}  \int_\UU f^q \varphi^q m_A^q \  \frac1{\delta(x)^{1/2}} ( n_x \cdot \hat v)^2    
  \\
 &&\quad \le   \frac14   \int_\UU f^q m_A^q \varphi^q \,  v \cdot \nabla_x (n_x \cdot \widetilde v)  - q  \int_\UU f^q \varphi^q \widetilde m^q \varpi  
  - \int_\UU  f^q  \widetilde m^q \partial_t \varphi^q
 \\
 &&\qquad - 
 \frac{1 }{ 4 D^{1/2}}  \int_\UU f^q \varphi^q m_A^q   \delta(x)^{1/2} \hat v \cdot D_x n_x \hat v \
  \\
 &&\quad \le C_{\Omega,A}   \int_\UU  f^q  m^q \langle \varpi_-\rangle \varphi^q  + C_A \int_\UU f^q  m^q  | \partial_t \varphi^q| , 
    \eean 
where we have used that $\delta \in W^{2,\infty}(\Omega)$ and $\Omega $ is bounded.   
\end{proof}

Using an interpolation argument, we may write our previous weighted $L^q$ estimate in a more convenient way where the penalization of a neighborhood of the boundary is made clearer. In order to do this, we use the following 
interpolation estimate.

\begin{lem}\label{lem:EstimL2poids} 
We set $\beta  := (2(d+1))^{-1}  $. 
 For any function $g : \OO \to \R$, there holds 
\beqn\label{eq:EstimL2poids}
\int_{\OO}  \frac{g^2 }{ \delta^{\beta}} \lesssim 
\int_{\OO} (g \langle v \rangle)^2  \frac{(n_x \cdot \hat v)^{2} }{  \delta^{1/2}} + 
\int_{\OO}  |\nabla_v (g \langle v \rangle)|^2.
\eeqn
\end{lem}

\begin{proof}[Proof of Lemma~\ref{lem:EstimL2poids}] 
For $\eta,\varsigma > 0$,  we start by writing
\bean
\int_{\OO} g^2 \delta^{-2\eta}
&=& 
\int_{\OO} \frac{ g^2 }{  \delta^{ 2\eta}}    {\bf 1}_{(n_x \cdot v)^{2} > \delta^{2\varsigma}} 
+ \int_{\OO}   \frac{ g^2 }{  \delta^{ 2\eta}}  {\bf 1}_{|n_x \cdot v|\le \delta^{\varsigma}}  = : T_1 + T_2.
\eean
For the first term, we have 
\bean
T_1 
\le
\int_{\OO} g^2   {\bf 1}_{(n_x \cdot v)^{2} > \delta^{2\varsigma}} \frac{(n_x \cdot v)^{2} }{  \delta^{2\varsigma+2\eta}} 
\le 
\int_{\OO} g^2   \frac{(n_x \cdot v)^{2} }{  \delta^{1/2}} ,
\eean
by choosing $2\varsigma + 2\eta = 1/2$.
For the second term, we define $2^* := 2d/(d-2)$  the  Sobolev exponent in dimension $d \ge 3$, 
and we compute
\bean
T_2
&\le&  
\int_\Omega  \delta^{-2\eta}  \Bigl( \int_{\R^d}  (\langle v \rangle  g)^{2^*}  \Bigr)^{2/2^*}  \Bigl( \int_{\R^d} \langle v \rangle^{- d}    {\bf 1}_{ |n_x \cdot v | \le \delta^{\varsigma}}   \Bigr)^{2/d}
\\
&\lesssim&  
  \int_{\Omega} \delta^{-2\eta+2\varsigma/d}    \int_{\R^d} |\nabla_v (\langle v \rangle g)|^2,
\eean
where we have used the H\"older inequality in the first line and the Sobolev inequality in the second line together with the observation that 
 $\langle v \rangle^{-d}   \in L^\infty(\R;L^{1}(\R^{d-1}))$. 
Choosing $2\varsigma/d = 2\eta$, we get $\eta = (4(d+1))^{-1}  $  and we conclude to \eqref{eq:EstimL2poids}. 
\end{proof}

Gathering the estimates of Proposition~\ref{prop:TheLqEstim} and Lemma~\ref{lem:EstimL2poids}, we immediately obtain the following result.

\begin{prop}\label{prop:EstimBord}  
Let $q \in (0,1)$ and $m : \R^d \to (0,\infty)$ be a radially symmetric decreasing  weight function such that   $m^{\frac{q}{1-q}} |v| \in L^1(\R^d)$.
There exists  $C = C(q,m,\Omega)>0$  such that for any nonnegative solution $f$  to the KFP equation \eqref{eq:Kolmogorov}--\eqref{eq:KolmoBdyCond}--\eqref{eq:initialDatum} and any test function $0 \le \varphi \in \DD((0,T))$,  there holds  
$$
\int_\UU \frac{f^q }{ \delta^{\beta}}  \frac{m ^q  }{ \langle v \rangle^2}  \varphi^q \le 
 C \int_\UU  f^q m^q [|\partial_t \varphi^q| +  \langle \varpi_-\rangle \varphi^q],
$$
where $\beta  := (2(d+1))^{-1}  $ and   $\varpi := \varpi_{m^{-1},q}$ is defined in  \eqref{def:varpi}. 
\end{prop}
 
By particularizing the choice of $m$, we obtain a first boundary penalizing weighted $L^1-L^q$ estimate which will be convenient for our purpose in the next steps.

\begin{prop}\label{prop:EstimBordL1Lq} 
For any  $q \in ((d+1)/(d+2),1)$, for any nonnegative solution $f$  to the KFP equation \eqref{eq:Kolmogorov}--\eqref{eq:KolmoBdyCond}--\eqref{eq:initialDatum} and any test function $0 \le \varphi \in \DD((0,T))$,  there holds  
$$
\int_\UU \frac{f^q }{ \delta^{\beta}}  \frac{ \varphi^q }{ \langle v \rangle^{2 + (d+2)q(1-q)} }  \le 
 C T^{1-q} \| \varphi^q \|_{W^{1,\infty}(0,T)} \| f \|_{L^1(\UU)}^q , 
$$
with  $C = C(q,d,\Omega)>0$ and $\beta = (2(d+1))^{-1}$ defined just above. 
\end{prop}

\begin{proof}[Proof of Proposition~\ref{prop:EstimBordL1Lq}] 
We choose   $m := \langle v \rangle^{-(d+2) (1-q)}$ and we observe that $m^{\frac{q }{ 1-q}} \langle v \rangle \in L^1$ and 
 $\varpi_{m^{-1},q} \in L^\infty$. From Proposition~\ref{prop:EstimBord},  we thus get
$$
\int_\UU \frac{f^q }{ \delta^{\beta}}   \frac{ \varphi^q }{ \langle v \rangle^{2 + (d+2)q(1-q)} } \le 
 C \| \varphi^q \|_{W^{1,\infty}(0,T)}  \int_\UU  f^q  \langle v \rangle^{-(d+2) (1-q)q}. 
$$
On the other hand, using the H\"older inequality, we have 
$$
 \int_\UU  f^q  \langle v \rangle^{-(d+2) (1-q)q} \le  \left(\int_\UU f \right)^q \left(T |\Omega| \right)^{1-q} \left( \int_{\R^d} \langle v \rangle^{-(d+2) q} \right)^{1-q},
 $$
and the last integral if finite because $(d+2)q > d$. We conclude by just gathering the two estimates. 
\end{proof}

 \bigskip 
\subsection{A weak weighted $L^1-L^p$ estimate. } 

Taking advantage of a known $L^1-L^p$ estimate available for the KFP equation set in the whole space and thus in the interior of the domain, we deduce a downgrade weighted $L^1-L^p$ estimate.
 We define 
\beqn\label{def:frakW3}
\mathfrak W_3 := \left\{ \omega : \R^d \to (0,\infty)\, ; \, \omega_0 := \omega  /\langle v \rangle \in \mathfrak W, \   |\nabla \omega_0| \omega_0^{-1} \langle v \rangle^{-1} \in L^\infty(\R^d)  \right\}.
\eeqn
We may notice that $\omega := \langle v \rangle^k e^{\zeta \langle v \rangle^s} \in \mathfrak W_3 $
under the condition \eqref{eq:condition-example-in-W}.

\begin{prop}\label{prop:EstimLploc}
Assume that $p \in (1,1+1/(2d))$,  $\alpha > p$ and $\omega \in \mathfrak W_3$.
There exists 
some constant $C = C(\Omega,p,\alpha,\omega) \in (0,\infty)$ such that any solution $f$ to the KFP equation \eqref{eq:Kolmogorov}--\eqref{eq:KolmoBdyCond}--\eqref{eq:initialDatum} satisfies 
\beqn\label{eq:EstimLploc}
 \left\| f \varphi  \frac{\omega}{\langle v \rangle}  \delta^{\alpha/p} \right\|_{L^p(\UU)} 
\le C T^{1/p + 2d (1-1/p)}  \|  \varphi \|_{W^{1,\infty}(0,T)} \| f \omega \|_{L^1(\UU)},
\eeqn
for any $0 \le \varphi \in \DD((0,T))$ and any $T>0$.
\end{prop}

\begin{proof}[Proof of Proposition~\ref{prop:EstimLploc}] 
For $\chi \in \DD(\Omega)$ such that $0 \le \chi \le 1$, we define $0 \le \bar f := f \varphi \chi \omega_0$, which is a solution to the equation  
$$
\partial_t \bar f + v \cdot \nabla_x \bar f -  \Delta_v \bar f - (v + 2 \frac{\nabla_v  \omega_0 }{ \omega_0}) \cdot \nabla_v \bar f = F
$$
set on $(0,T) \times \R^d \times \R^d$, 
with 
$$
F :=    f \omega_0 (\varphi' \chi   + \varphi v \cdot \nabla_x\chi ) 
+ \bar f \Bigl( d - v \cdot \frac{\nabla_v \omega_0 }{ \omega_0} +   2 \frac{|\nabla_v \omega_0|^2 }{ \omega_0^2} - \frac{\Delta_v \omega_0 }{ \omega_0}  \Bigr) .
$$
Because $\omega_0 \in \mathfrak W$, we have 
$$
F_+ \le f \omega_0  \langle v \rangle (|\varphi'|\chi   + \varphi |\nabla_x\chi| ) + f  \varphi \chi \omega_0   \kappa_{\omega_0}.
$$
From  \cite[Theorem~1.5]{francesca2021fundamental} for instance and because $ |\nabla \omega_0| \omega_0^{-1} \lesssim \langle v \rangle$, we know that 
$$\ 
\bar f 
\le  \int_0^t K_{t-s}  \star F_{+s} \, ds,  
$$
 where $\star=\star_{x,v}$ stands for a convenient convolution operation and $K_\tau$ is the Kolmogorov kernel defined by
 $$
K_\tau(x,v) := \frac{C_1 }{ \tau^{2d}} \exp \left( - \frac{3C_2 }{ \tau^3} |x- \frac{\tau}2v|^2 - \frac{C_2 }{ 4\tau}|v|^2 \right), \quad C_i > 0.
$$
We next compute  
\bean
\| \bar f \|_{L^p([0,T] \times \R^{2d})}^p 
&\le& \int_0^T \bigl\| \int_0^t K_{t-s} \star F_{+s} \bigr\|_{L^p(\R^{2d})}^p dt
\\
&\le& \|  K    \|_{L^p([0,T] \times \R^{2d})}^p  \|  F_+ \|_{L^1([0,T] \times \R^{2d})}^p , 
\eean
and because $1 \le p < 1+1/(2d)$, we find
$$
\|  K    \|_{L^p([0,T] \times \R^{2d})}^p = C_{K,p}  T^{1 - 2d (p-1)}.
$$
As a consequence, we have 
\beqn\label{eq:EstimStep1-prop:EstimLploc}
 \| f \varphi \omega_0 \chi \|_{L^{p}(\UU)}  
\lesssim  C_T  \| \varphi \|_{W^{1,\infty}}  \| \chi \|_{W^{1,\infty}}   \|  f  \omega   \|_{L^1(\UU)},
\eeqn
with $C_T:=T^{1/p + 2d (1-1/p)}$.

\smallskip\noindent
{\sl Step 2.} 
We define $\Omega_k := \{ x \in \Omega \mid \delta(x) > 2^{-k} \}$ and we choose $\chi_k \in \DD(\Omega)$ such that ${\mathbf 1}_{\Omega_{k+1}} \le \chi_k \le {\mathbf 1}_{\Omega_k}$ and  $2^{-k}\| \chi_k \|_{W^{1,\infty}} \lesssim 1$ uniformly in $k \ge 1$. 
We also denote $\UU_k := (0,T) \times \Omega_k \times \R^d$.
We deduce from \eqref{eq:EstimStep1-prop:EstimLploc} that 
$$
\| f \varphi  \omega_0 \|_{L^p(\UU_{k+1})}  \lesssim  2^k C_T \| \varphi   \|_{W^{1,\infty}(0,T)}  \| f \omega \|_{L^1(\UU)}, \quad \forall \, k \ge 1.
$$
Summing up, we obtain 
\bean
\int_\UU \delta^{\alpha}  (\varphi f\omega_0) ^{p} 
&=& \sum_k \int_{\UU_{k+1} \backslash \UU_k} \delta^{\alpha}  (\varphi  f \omega_0)^p
\\
&\lesssim& \sum_k  2^{-k   \alpha} \int_{\UU_{k+1}}  (\varphi f\omega_0)^p
\\
&\lesssim&  \sum_k 2^{ k(p-  \alpha)} C_T^p \| \varphi   \|_{W^{1,\infty}} ^p \|   f\omega \|_{L^1(\UU)}^p
\\
&\lesssim& C_T^p \| \varphi   \|_{W^{1,\infty}(0,T)}^p \| f \omega  \|_{L^1(\UU)}^p, 
\eean
because $\alpha > p$, what is nothing but \eqref{eq:EstimLploc}.
\end{proof}

\subsection{The $L^1-L^r$ estimate up to the boundary}

We start with a classical interpolation result.

\begin{lem}\label{lem:interpolation}
For any weight functions $\sigma_i : \UU \to (0,\infty)$ and any exponents $0 < r_0 < r_1 < \infty$, 
$0 < \theta < 1$,  there holds
$$
\| g \|_{L^r_\sigma} \le \| g  \|_{L^{r_0}_{\sigma_0}}^{1-\theta} \| g\|_{L^{r_1}_{\sigma_1}}^\theta, 
$$
with $1/r := (1-\theta)/r_0 + \theta/r_1$ and $\sigma := \sigma_0^{1-\theta} \sigma_1^\theta$.  
\end{lem}

We include the very classical proof because the statement is usually written assuming rather $1 \le  r_0 < r_1 < \infty$, but that last restriction is not needed.

\begin{proof}[Proof of Lemma~\ref{lem:interpolation}]
We write 
\bean
\Bigl(\int f^r \sigma^r \Bigr)^{1/r} 
&=&   \Bigl(\int (f \sigma_0)^{r(1-\theta)} (f \sigma_1)^{r\theta} \Bigr)^{1/r} 
\\
&\le&   \Bigl(\int (f \sigma_0)^{a'r(1-\theta)}  \Bigr)^{1/a'r}  \Bigl(\int   (f \sigma_1)^{ar\theta} \Bigr)^{1/ar} 
\eean
thanks to the H\"older inequality with  $\frac1a := \frac{\theta r}{r_1 } = 1 - (1 - \theta)   \frac{r}{r_0} < 1$, from what we immediately deduce $r_1 = ar\theta$ and $r_0 = a'r(1-\theta)$, and thus conclude. 
\end{proof}

\medskip
We are now in position of stating our weighted $L^1-L^r$ estimate up to the boundary which is the well-known cornerstone step in the proof of DeGiorgi-Nash-Moser gain of integrability estimate.
 
\begin{prop}\label{prop:EstimL1Lr-omega} There exist an exponent $r > 1$ and some constants $\eta > 0$, $\theta,q \in (0,1)$ such that 
 any solution $f$ to the KFP equation \eqref{eq:Kolmogorov}-\eqref{eq:KolmoBdyCond} satisfies 
\beqn\label{eq:EstimL1Lr-omega}
 \| \varphi  f  \omega^\sharp \|_{L^r(\UU)} \le C T^\eta \| \varphi^q \|_{W^{1,\infty}(0,T)}^{1/q}  \| f\omega \|_{L^1(\UU)},
\eeqn
for any weight function $\omega \in  \mathfrak W_3$ and any test function $0 \le \varphi \in \DD((0,T))$, with $\omega^\sharp := \omega^\theta  \langle v \rangle^{-4}$ and $C = C(d,\Omega,\omega)$. 
A possible choice is $\theta = \theta_1 := (2d+3)^{-1}$. 
\end{prop} 

\begin{proof}[Proof of Proposition~\ref{prop:EstimL1Lr-omega}]
From Proposition~\ref{prop:EstimBordL1Lq}, we have 
$$
\bigl\| f \varphi  \frac{1 }{ \delta^{\beta/q}} \frac{1}{ \langle v \rangle^{2/q + (d+2)(1-q)} }  \Bigr\|_{L^q(\UU)} \le 
 C T^{1/q-1} \| \varphi^q \|_{W^{1,\infty}(0,T)}^{1/q} \| f \omega \|_{L^1(\UU)}, 
$$ 
for some exponent $q \in  ((d+1)/(d+2),1)$ and with $\beta  := (2(d+1))^{-1}$. 
Together with Proposition~\ref{prop:EstimLploc} and Lemma~\ref{lem:interpolation}, we deduce that 
$$
\| f \varphi \sigma \|_{L^r} \le  C T^\eta \| \varphi^q \|_{W^{1,\infty}}^{1/q} \| f \omega \|_{L^1(\UU)}, 
$$
for any $\theta \in (0,1)$ with 
$$
\frac{1 }{ r} = \frac{1-\theta }{ q} + \frac{\theta }{ p}, \quad
\sigma := \frac{\delta^{\alpha \theta/p} }{ \delta^{(1-\theta)\beta/q}} \frac{\omega^\theta }{ \langle v \rangle^{\theta + (2/q+(d+2)(1-q))(1-\theta)}},
$$
and
$$
\eta := (1-\theta) (1/q-1) + \theta (1/p + 2d (1-1/p)), 
$$
where we recall here that $p \in (1,1+1/(2d))$ and  $\alpha > p$ are arbitrary. We first choose 
$$
\theta   := \frac{\beta/q }{ \beta/q+\alpha/p} =  {1 \over 1 +2(d+1)q \alpha/p}
, 
$$
in such a way that $\delta^{\alpha \theta/p-(1-\theta)\beta/q} \equiv 1$. 
 In order to track the dependency in both the exponent and the weight function, we  choose $\alpha := p/q$, so that  $\theta = \theta_1 := (1+2(d+1))^{-1}$, and because 
 $r=r_q \to r_*$ as $q \to 1$ with
$$
\frac{1 }{ r_*} = 1-\theta_1 + \frac{\theta_1 }{ p} < 1,
$$
we may choose $q \in ((d+1)/(d+2),1)$ large enough in such a way that $r > 1$. We finally observe that $2/q+(d+2)(1-q) \le 4$ so that $\sigma \gtrsim \omega^\sharp$. 
\end{proof}

\subsection{The $L^1-L^p$ estimate on the dual problem } 

 We consider the dual  backward problem \eqref{eq:dualKFP} for which we establish the same kind of estimate as for the forward KFP  problem  \eqref{eq:Kolmogorov}-\eqref{eq:KolmoBdyCond}.
In all this section, we denote by $q \in (0,1)$ the exponent chosen during the proof of Proposition~\ref{prop:EstimL1Lr-omega} and we  define  
\bear\label{def:frakN1}
  \mathfrak N_1  :=  \Bigl\{ m := e^{-\zeta |v|^2 }, \ \zeta \in (0,1/2) \Bigr\}.
\eear

\begin{prop}\label{prop:EstimL1Lr-m} There exist some exponent $r > 1$ and some constants $\eta > 0$ such that for any weigh function $m \in \mathfrak N_1$ and  
 any solution $g$ to the  dual  backward problem \eqref{eq:dualKFP}, there holds 
\beqn\label{eq:EstimL1Lr-m}
 \| \varphi  g  m' \|_{L^{r}(\UU)} \lesssim T^{\eta_1} \| \varphi^q \|_{W^{1,\infty}(0,T)}^{1/q}  \| g m \|_{L^1(\UU)},
\eeqn
for  any test function $0 \le \varphi \in \DD((0,T))$ and some exponential weight function $m' := \exp (- \zeta' |v|^2)$, with $\zeta' \in (\zeta,1/2]$. 
\end{prop} 

We emphasize that the exponent $r > 1$ can be taken identically  as in Proposition~\ref{prop:EstimL1Lr-omega}, and for the sake of simplicity it is what we will do in the sequel. 

\begin{proof}[Proof of Proposition~\ref{prop:EstimL1Lr-m}] 
The proof follows the same steps as for the proof of Proposition~\ref{prop:EstimL1Lr-omega} and we thus repeat it without too much details. 
 
\smallskip\noindent
{\sl Step 1. Boundary penalizing $L^1-L^q$ estimate.} 
From \cite[Lemma~7.7]{sanchez:hal-04093201} or a direct computation, we have 
\beqn\label{eq:C*ffq-1}
\int (\CCC^* g ) \, g^{q-1} m^q = -  \frac{4 (q-1)}{ q^2} \int |\nabla_v (gm)^{q/2} |^2   + \int g^q m^q \wp_{m,q},
\eeqn
with  $\CCC^* $ defined in  \eqref{def:CCC*} and 
\beqn\label{eq:wp}
\wp_{m,q} := 2 \left(1 - \frac{1}{q} \right) \frac{|\nabla_v m|^2 }{ m^2}  + \left( \frac 2q -1 \right) \frac{ \Delta_v m }m  
 +\frac{d}{ q}   + v \cdot \frac{\nabla_v m }{ m} . 
\eeqn
Considering a solution $g$ to the  dual  backward problem \eqref{eq:dualKFP}  and $q \not=1$, 
we may write 
\beqn\label{eq:TT*gq} 
\TT^* \frac{g^q }{ q} = g^{q-1} \TT^* g = g^{q-1} \CCC^* g,
\eeqn
with $\TT^*$  defined in \eqref{def:TT*}.
We define the  modified weight function  $\frak M$ by 
 $$
\frak M := \MMM \left(1  - \frac1{4D^{1/2}}    \delta(x)^{1/2} n_x \cdot \widetilde v \right).
$$
Multiplying the equation \eqref{eq:TT*gq} by $\Phi^q := \varphi^q\frak M $ with $\varphi \in \DD(0,T)$, and integrating in all the variables, we obtain
\bean
- \frac1q\int_\Gamma (\gamma g)^q \Phi^q  n_x \cdot v + \frac1q \int_\UU g^q \TT \Phi^q
= \int_\UU g^{q-1} (\CCC^* g)   \Phi^q, 
\eean
with $\TT$ defined in \eqref{def:TT}. 
Together with \eqref{eq:C*ffq-1}, we thus deduce 
\beqn\label{eq:weakDualKolmogorovLq}
 4\frac{1-q }{ q^2}   \int_\UU |\nabla_v (g\Phi)^{q/2}|^2 + \frac1q\int_\Gamma (\gamma g)^q \Phi^q  n_x \cdot v
=   \frac1q \int_\UU g^q \TT \Phi^q - \int g^q \Phi^q \wp, 
\eeqn
with $\wp = \wp_{\frak M^{1/q}, q}$.
For the  boundary term at the LHS, we argue similarly as during the proof of Lemma~\ref{lem:B11:lem0}. 
We   observe that 
\bean
 \int_\Sigma (\gamma g)^q \MMM n_x \cdot v
\ge
\int_{\Sigma_+} \iota (\widetilde{ \gamma_- g \MMM})^q  \MMM (n_x \cdot v)_+ - \int_{\Sigma_-} \iota (\gamma_-g)^q \MMM  (n_x \cdot v)_-
\ge 0, 
\eean
where we have used the concavity of the function $G \mapsto G^q$ and we have removed the contribution of the specular reflection in the first inequality, and we have used H\"older's inequality 
$$
\int_{\R^d} \gamma_- g^q \MMM  (n_x \cdot v)_-
\le \bigl( \widetilde{ \gamma_- g \MMM} \bigr)^q \Bigl( \int_{\R^d} \MMM  (n_x \cdot v)_-\Bigr)^{1-q} 
$$
and the normalization condition \eqref{eq:FPK-def_M} in the second inequality.  
 We may then proceed exactly as in the proof of Proposition~\ref{prop:TheLqEstim}, and we obtain  
$$
\int g^q \frak M \, \,  \frac{(n_x \cdot \hat v)^{2} }{  \delta^{1/2}} \varphi^q+  
\int  |\nabla_v(g^{q/2} \frak M^{1/2})|^2  \varphi^q
\le 
\frac{C_\Omega }{ 1-q}   \int g^q  \frak M [|\partial_t \varphi^q| +  \varphi^q \langle \wp_- \rangle].
$$
As in Proposition~\ref{prop:EstimBord} and with the help of the interpolation  Lemma~\ref{lem:EstimL2poids}, we deduce 
\beqn\label{eq:Estim-g1}
\int \frac{g^q }{ \delta^\beta}  \frac{ \frak M   }{ \langle v \rangle^2}  \varphi^q 
\le 
\frac{C_\Omega }{ 1-q}   \int g^q \frak M [|\partial_t \varphi^q| +  \varphi^q \langle \wp_- \rangle], 
\eeqn
 for the same $\beta  := (2(d+1))^{-1}  $. 
Finally, using that $ \langle \wp_- \rangle \lesssim \langle v \rangle^2$ and arguing similarly as in the proof of Proposition~\ref{prop:EstimBordL1Lq} with the help of a last Holder inequality for handling the RHS term, we get  
\beqn\label{eq:Estim-g1BIS}
\int \frac{g^q }{ \delta^\beta}  \frac{ \frak M }{ \langle v \rangle^2}   \varphi^q 
\le C T^{1-q}
 \| \varphi^q  \|_{W^{1,\infty}}\| g \frak M^{1/q}  \langle v \rangle^{2/q} \langle v \rangle^{(1-q)(d+1)/q} \|_{L^1(\UU)}^q, 
\eeqn
for some constant $C = C(q,\Omega) > 0$.

\smallskip\noindent
{\sl Step 2. Weak weighted $L^1-L^p$ estimate, $p>1$.} Consider  a solution $g$ of the dual problem \eqref{eq:dualKFP}, 
$0 \le \varphi \in \DD((0,T))$, $0 \le \chi \in \DD(\Omega)$ and a weight function $m \in \frak N_1$, so that $m_0 := m \langle v \rangle^{-2}$ satisfies 
$$
  \frac{|\nabla m_0|^2 }{ m_0^2}  +  \frac{|\Delta m_0| }{ m_0}   \lesssim \langle v \rangle^2 .
$$
We set  $\bar g := g \varphi \chi m_0$ and we easily compute 
$$
- \partial_t \bar g - v \cdot \nabla_x \bar g - \Delta_v \bar g  + (2 \frac{\nabla_v m_0 }{ m_0} - v ) \cdot \nabla_v \bar g = G, 
$$
with 
$$
G : =\bar g \Bigl[ 2 \frac{|\nabla_v m_0|^2 }{ m_0^2} - \frac{\Delta_v m_0 }{ m_0} + v \cdot \frac{\nabla_v m_0 }{ m_0} \Bigr] - gm_0 (\partial_t +v \cdot \nabla_x)(\varphi \chi). 
$$
Proceeding as in the proof of Proposition~\ref{prop:EstimLploc}, we get first 
\bean
\| \bar g \|_{L^{p}(\R^{2d+1})}  
\le C T^{\eta_2}  \| \varphi \|_{W^{1,\infty}}  \| \chi \|_{W^{1,\infty}}    \|  g m    \|_{L^{1}(\UU)} , 
\eean
for any $p \in (1,1+1/(2d))$ and with $\eta_2 := 1/p + 2d (1-1/p)$.
By interpolation, we then conclude 
\beqn\label{eq:Estim-g2}
 \|  g \varphi    \frac{m }{ \langle v \rangle^2} \delta^{\alpha/p} \|_{L^p(\UU)} 
\le C T^{\eta_2}   \|  \varphi \|_{W^{1,\infty}} \| g m \|_{L^1(\UU)},
\eeqn
for any $\alpha > p$ and some constant $C = C(\alpha, \Omega, m) > 0$. 

\smallskip\noindent
{\sl Step 3. Weighted $L^1-L^r$ estimate, $r > 1$.} 
We consider again a weight function $m \in \frak N_1$, 
and we observe that  $m \ge \frak M^{1/q}  \langle v \rangle^{2/q} \langle v \rangle^{(1-q)(d+1)/q}$. 
From Step~1, we thus find 
$$
\bigl\| \frac{ g }{ \delta^{\beta/q} } \frac{\MMM}{ \langle v \rangle^{2/q}} \varphi \bigr\|_{L^q (\UU)} 
\le C T^{\eta_1}
 \| \varphi^q  \|_{W^{1,\infty}}^{1/q} \| g m \|_{L^1(\UU)}, 
$$
with $\eta_1 := 1/q-1$.  Since $m$ satisfies the requirement of Step~2, we may thus use the above estimate together with \eqref{eq:Estim-g2} and the interpolation Lemma~\ref{lem:interpolation} as during the proof of Proposition~\ref{prop:EstimBordL1Lq}. We get
$$
\| g \varphi \sigma \|_{L^r(\UU)} \le  C T^\eta \| \varphi^q \|_{W^{1,\infty}}^{1/q} \| gm \|_{L^1(\UU)}, 
$$
with 
$$
\frac{1 }{ r} = \frac{1-\theta_1 }{ q} + \frac{\theta_1 }{ p}, \quad
\sigma := \frac{ m^{\theta_1} \MMM^{1-\theta_1}}{ \langle v \rangle^{2\theta_1 +  (2/q)(1-\theta_1)}},
$$
and
$$
\eta := (1-\theta_1) (1/q-1) + \theta_1 (1/p + 2d (1-1/p)),
$$
where we have fixed 
$$
p \in (1,1+1/(2d)), \quad   \alpha :=  p/q, \quad 
\theta_1 :=(2d+3)^{-1}. 
$$ 
Because of the choice of $q$ (large enough), we have  $r > 1$. On the other hand, we clearly have $\sigma \gtrsim m' := \exp(-\zeta' |v|^2)$, for any
$\zeta' \in (\theta_1 \zeta + (1-\theta_1)/2,1/2)$, in particular $\zeta' > \zeta$.
\end{proof}

\subsection{Conclusion of the proof}
\label{subsec:conclusion}

 We now conclude the proof of  Theorem~\ref{th:DGNML1Linfty} in several elementary and classical (after Nash's work) steps.

\begin{proof}[Proof of  Theorem~\ref{th:DGNML1Linfty}]
We split the proof into
four steps. We denote by  $r >1$ the (same) exponent defined in Propositions~\ref{prop:EstimL1Lr-omega} and \ref{prop:EstimL1Lr-m}.
We define 
$$
\frak W_1 := \{ \exp (\zeta' |v|^2); \  \zeta' \in ((1-\theta_1)/2,1/2)  \} 
$$
and we fix $\omega \in   \frak W_1$.

 \medskip\noindent
 {\sl Step 1.} Take $\omega_1 := \omega$, so that 
 $\omega_1 \in \mathfrak W_0 \cap\mathfrak W_3 $ and $\omega_r := \omega^\sharp = \omega^{\theta} \langle v \rangle^{-4} \in \mathfrak W_0$, where 
 $\theta \in (0,1)$ is defined in the statement of Proposition~\ref{prop:EstimL1Lr-omega}. We  claim that there exist $\nu_1, \kappa_1 > 0$  such that 
\beqn\label{eq:proofTh1-step1}
 T^{\nu_1} \| S_\LLL(T) f_0 \|_{L^{r}_{\omega_r}(\OO)} \lesssim e^{\kappa_1 T} \| f_0 \|_{L^1_{\omega_1}(\OO)}, \quad \forall \, T > 0, \ \forall \, f_0 \in L^1_{\omega_1}(\OO).
\eeqn
 We set $f_t := S_\LLL(t) f_0$. 
 On the one hand, from  Proposition~\ref{prop:EstimLp} with $p=r$, we have 
 \bean
\frac{T}2 \| f_T \|_{L^r_{\omega_r}}^r 
&\lesssim&  \int_{T/2}^{T}   e^{r\kappa (T-t)} \| f_t \|_{L^r_{\omega_r} }^r dt
\\
&\lesssim& e^{r\kappa T}  \int_{0}^{T}    \| f_t \varphi_0(t/T) \|_{L^r_{\omega_r} }^r dt, 
\eean
with $\varphi_0 \in C^1_c((0,2))$,  ${\mathbf 1}_{[1/2,1]} \le \varphi_0 \le 1$,  $\varphi_0^q \in W^{1,\infty}$, $q \in ((d+1)/(d+2),1)$. On the other hand, thanks to Proposition~\ref{prop:EstimL1Lr-omega} applied with $\varphi(t) := \varphi_0(t/T)$ and next to
 Proposition~\ref{prop:EstimLp} with $p=1$, we deduce 
 \bean
\frac{T}2 \| f_T \|_{L^r_{\omega_r}}^r 
 &\lesssim& e^{r\kappa T} T^{r\eta} \bigl( 1 + \frac{1 }{ T} \bigr)^{r/q}   \Bigl( \int_0^{T}  \| f_t \|_{L^1_{\omega_1} } dt \Bigr)^r
\\
 &\lesssim& e^{r\kappa T} T^{r\eta} \bigl( 1 + \frac{1}{T} \bigr)^{r/q}   \Bigl( \int_0^{T}e^{\kappa  t}  dt \Bigr)^r  \| f_0 \|_{L^1_{\omega_1} }^r,
 \eean
from what \eqref{eq:proofTh1-step1} follows with $\nu_1 := 1/r - \eta - 1/q$ and any $\kappa_1 \ge 3\kappa$.
 
 \medskip\noindent{\sl Step 2.}   Take $m_1  := \exp(-\zeta |v|^2)$, $\zeta \in (0,1/2)$,  and $m_r:= \exp(-\zeta' |v|^2)$, 
$\zeta' \in (\theta_1 \zeta + (1-\theta_1)/2,1/2)$ as in the statement of Proposition~\ref{prop:EstimL1Lr-m}. 
We emphasize that $m_r^{-1} \in \frak M_{0}$ as defined in \eqref{def:frakW0} and $m_1 \in \frak N_{0}$ as defined in \eqref{def:frakM0}.
 We now claim that there exist $\nu_2, \kappa_2 > 0$  
 such that 
\beqn\label{eq:proofTh1-step2}
 T^{\nu_2} \| S^*_\LLL(T) g_0 \|_{L^r_{m_r}(\OO)} \lesssim e^{\kappa_2 T} \| g_0 \|_{L^1_{m_1}(\OO)}, \quad \forall \, T > 0, \ \forall \, g_0 \in L^1_{m_1}(\OO).
\eeqn
We repeat the argument presented in Step 1. We set $g_t := S^*_\LLL(t) g_0$. 
 On the one hand, from the dual counterpart of Proposition~\ref{prop:EstimLp} with $p=r'$ and next from Proposition~\ref{prop:EstimL1Lr-m}, we have 
 \bean
\frac{T}2 \| g_T \|_{L^r_{m_r}}^r 
&\lesssim& e^{r\kappa T}  \int_{0}^{T}    \| g_t \varphi_0(t/T) \|_{L^r_{m_r} }^r dt, 
\\
&\le& 
 e^{r\kappa T} T^{r\eta} \bigl( 1 + \frac{1}{T} \bigr)^{r/q}   \Bigl( \int_0^{T}  \| g_t \|_{L^1_{m_1} } dt \Bigr)^r,
\eean
where $\varphi_0$ is the same function as above. We conclude to \eqref{eq:proofTh1-step2} thanks to Lemma~\ref{lem:SLestimLinfty}.

 \medskip\noindent{\sl Step 3.}  Observing that for $\omega := \exp(\zeta' |v|^2)$,  $ \zeta' \in ((1-\theta_1)/2,1/2)$, there exists $\zeta \in (0,1/2)$ such that $\zeta' \in (\theta_1 \zeta + (1-\theta_1)/2,1/2)$,
 the dual counterpart of \eqref{eq:proofTh1-step2} writes 
 \beqn\label{eq:proofTh1-step3}
 T^{\nu_2} \| S_\LLL(T) f_0 \|_{L^\infty_{\omega_\infty}(\OO)} \lesssim e^{\kappa_2 T} \| f_0 \|_{L^s_{\omega}(\OO)}, \quad \forall \, T > 0, \ \forall \, f_0 \in L^s_{\omega}(\OO), 
\eeqn
with $\omega_\infty := m_1^{-1} = \exp(\zeta |v|^2)$ and $s := r' \in (1,\infty)$. 
 Interpolating \eqref{eq:proofTh1-step1} and  \eqref{eq:proofTh1-step3}, for any $1 \le p < q \le \infty$, we obtain
 $$
\| S_\LLL(T) f_0 \|_{L^q_{\omega_q}} \le C_1 \frac{e^{C_2 T} }{ t^{\nu(1/p-1/q)}}   \| f_0  \|_{L^p_{\omega}},  \quad \forall \, T > 0, \ \forall \, f_0 \in L^p_{\omega}(\OO), 
$$
with $\nu := \max(\nu_1,\nu_2) (1-1/r)^{-1}$, $C_2 := \max(\kappa_1,\kappa_2)$,  and the appropriate interpolated weight function $\omega_q$,
in particular $\omega_q \ge \omega' := \exp(\zeta''|v|^2)$, with $\zeta'' := \min(\theta_1 \zeta',\zeta)$. 
\end{proof}

\section{Hypocoercivity: Proof of Theorem~\ref{theo:hypo}}
   \label{sec:proofTh2}
   
  We adapt the proof of \cite[Theorem~1.1]{MR4581432}. 
 We start introducing some notations and recalling some classical results about the Poisson equation.
 For any convenient function or distribution $\xi  : \Omega \to \R$, we define   $u := (-\Delta_x)^{-1} \xi  : \Omega \to \R$ as the associated  solution to the Poisson equation with Neumann condition. 
More precisely, for any $\eta_i \in L^2(\Omega)$, $\langle \eta_1 \rangle = 0$, we define $u \in H$, with $H := \{ u \in H^1(\Omega), \,\, \langle u \rangle = 0 \}$, as the solution of 
the variational problem
\beqn\label{eq:Poisson-VarFormulation}
\int_\Omega \nabla_x u \cdot \nabla_x w  = 
\int_\Omega  \{ w \eta_1 -    \nabla_x w \cdot  \eta_2\} , \quad \forall \, w \in H,
\eeqn
which is indeed a variational solution to the 
Poisson equation with Neumann condition
\beqn\label{eq:Poisson-PDEFormulation}
- \Delta_x u = \eta_1 + \Div_x  \eta_2  \ \ \hbox{in}\ \ \Omega, \quad  n_x \cdot (\nabla_x u - \eta_2) = 0   \ \ \hbox{on}\ \ \partial\Omega.
\eeqn
It is well-known that the above variational problem has a unique solution thanks to the Poincar\'e-Wirtinger inequality and the Lax-Milgram Theorem, that  
\beqn\label{eq:PoissonNH1}
\|u \|_{H^1(\Omega)} \lesssim \sum_{i=1}^2 \| \eta_i \|_{L^2(\Omega)}, 
\eeqn
holds true and that  the additional regularity estimates 
\beqn\label{eq:PoissonNH2}
\|u \|_{H^1(\partial\Omega)} \lesssim \| u \|_{H^2(\Omega)} \lesssim \| \eta_1 \|_{L^2(\Omega)}
\eeqn
holds when $\eta_2 = 0$. 
We define 
$$
\HH := L^2(\mu^{-1} dvdx), \quad 
\HH_0 := \{ f \in \HH; \ \la\!\la f \ra\!\ra = 0 \}, 
$$
where $\mu$ is defined in \eqref{eq:FPK-def_mu} and $ \la\!\la \cdot \ra\!\ra$ in \eqref{eq:intro_mass_cons}.
We next define the new (twisted) scalar product $(\!( \cdot, \cdot )\!)$ on~$\HH_0$ by 
$$
(\!( f , g )\!) 
 := ( f , g )_{\HH} 
+ \eps (\nabla_x (-\Delta_x)^{-1} \varrho_f, j_g   )_{L^2} + \eps (\nabla_x (-\Delta_x)^{-1} \varrho_g, j_f   )_{L^2},
$$
with $ \eps > 0$ small enough to be fixed later, $L^2 := L^2_x(\Omega)$ and where the mass $\varrho_f$ and the momentum $j_f$ are defined respectively by
$$
\varrho_h(x) = \varrho[h](x) := \langle h \rangle, \quad j_h(x) =j[h](x) := \langle h v \rangle, \quad \langle H \rangle := \int_{\R^d} H (x,v) \, dv.
$$
For any $f \in \HH_0 $, we next decompose 
\beqn\label{eq:fdecomposition}
f = \pi f + f^\perp,
\eeqn
with the macroscopic part $\pi f$ given by
$$
\pi f (x,v) = \varrho_f(x) \mu(v),
$$
and we remark that 
\beqn\label{eq:fdecompositionNorm}
\| f \|_{\HH}^2 
= \| f^\perp \|_{\HH}^2 + 
\| \pi f \|_{\HH}^2, \quad 
\| \pi f \|_{\HH}^2 
= \| \varrho_f \|_{L^2}^2.
\eeqn
as well as 
\beqn\label{eq:fdecompositionNormEstim}
 \| \varrho_f \|_{L^2} \le \| f \|_\HH, \quad  \| j_f \|_{L^2} \lesssim \| f^\perp \|_\HH \le  \| f \|_\HH .
 \eeqn
It is worth emphasizing that 
 \bean
 | (\nabla_x (-\Delta_x)^{-1} \varrho_f, j_f   )_{L^2} | 
 &\le&  \| \nabla_x (-\Delta_x)^{-1} \varrho_f \|_{L^2} \| j_f \|_{L^2}
 \\
 &\lesssim&  \| \varrho_f \|_{L^2} \| f^\perp \|_{\HH} \lesssim  \| f \|_\HH^2, 
 \eean
 from  the Cauchy-Schwarz inequality, \eqref{eq:PoissonNH2} and \eqref{eq:fdecompositionNormEstim}. 
Denoting by $\Nt \cdot \Nt$ the norm associated to the scalar product $(\!( \cdot , \cdot )\!)$, we in particular deduce that 
\beqn\label{eq:equivNorm}
\| f \|_{\HH} \lesssim \Nt f \Nt \lesssim \| f \|_{\HH}, \quad \forall \, f \in \HH_0.
\eeqn
We finally define the Dirichlet form associated to the operator $\LLL$ defined in \eqref{def:LLL} for the twisted scalar product
$$
D[f] := (\!( - \LLL f, f ) \!), \quad f \in \HH_0.
$$
More explicitly, we have 
$$
D[f] = D_1[f] + D_2[f] + D_3[f],
$$
with 
$$
D_1[f] := (-\LLL f,f)_\HH, \quad D_2[f] := \eps (\nabla_x \Delta_x^{-1} \varrho_f, j[\LLL f])_{L^2}, 
\quad D_3[f] :=  \eps (\nabla_x \Delta_x^{-1} \varrho[\LLL f], j_f)_{L^2}, 
$$
and we estimate each term separately. 
For simplicity we introduce the notations 
$\DDD^\perp := \mathrm{Id} - \DDD$, where we recall that $\DDD$ is given by 
\eqref{eq:FPK-def_D}  and $\partial \HH_+ := L^2(\Sigma_+ ; \mu^{-1}(v) n_x \cdot v dv d\sigma_{\! x})$. 
It is worth emphasizing that because $f \in \mathrm{Dom}(\LLL)$, the trace functions $\gamma_{\pm} f $ are well defined. We refer the interested reader to \cite{MR2721875,sanchez:hal-04093201,CM-Landau**} and the references therein for a suitable definition of the trace function for solutions to the KFP equation.

\medskip
We estimate the first term involved in the Dirichlet form $D$.

\begin{lem}\label{lem:micro}
For any $f \in \HH$, there holds  
$$
( - \LLL f , f )_{\HH} \ge   \| f^\perp \|_{\HH}^2 + \frac12 \| \sqrt{\iota(2-\iota)} \DDD^\perp \gamma_{+} f \|_{\partial \HH_+}^2.
$$   
\end{lem}

\begin{proof}[Proof of Lemma~\ref{lem:micro}]
Recalling \eqref{def:LLL} and \eqref{def:CCCf}, we write 
$$
( - \LLL f , f )_{\HH} = ( - \CCC f , f )_{\HH} 
+( v \cdot \nabla_x f , f )_{\HH}.
$$
On the one hand, we recall the classical Poincaré inequality 
$$
\| h - \langle h  \mu\rangle \|_{L^2(\mu)} \le \| \nabla_v h \|_{L^2(\mu)}, \quad \forall \, h \in L^2(\mu dvdx), 
$$
from what we classically deduce 
\bean
 ( - \CCC f , f )_{\HH} 
 &=&- \int_\OO \Div_v(\mu \nabla_v (f/\mu)) f /\mu \, dvdx
 \\
  &=& \int_\OO  |\nabla_v (f/\mu)|^2 \mu \, dv \, dx
 \\
 &\ge&  \int_\OO  |f/\mu - \langle f \rangle |^2 \mu \, dv \, dx =  \| f^\perp \|^2_\HH.
 \eean
The second part of the estimate has been proved during the proof of \cite[Lemma 3.1]{MR4581432}. 
\end{proof}

We recall the identity established in \cite[Lemma 3.2]{MR4581432}. 

\begin{lem}\label{lem:boundary}
Let  $\phi: \R^d \to \R$. For any $x \in \partial\Omega$, there holds
$$
\begin{aligned}
\int_{\R^d}  \phi(v) \gamma f(x,v) \, n_x \cdot v \, dv  
& = \int_{\Sigma^x_{+}}  \phi(v) \iota(x) \DDD^\perp \gamma_{+} f  \, n_x \cdot v \, dv  \\
&\quad 
+ \int_{\Sigma^x_{+}}  \left\{ \phi(v) - \phi(\VV_x v) \right\}  (1-\iota(x)) \DDD^\perp \gamma_{+}  f  \, n_x \cdot v \, dv  \\
&\quad 
+ \int_{\Sigma^x_{+}}   \left\{ \phi(v) - \phi(\VV_x v) \right\}  \DDD \gamma_{+} f  \, n_x \cdot v \, dv .
\end{aligned}
$$    
\end{lem}

\medskip
We estimate the second term involved in the Dirichlet form $D$.

\begin{lem}\label{lem:mass1}
There is a constant $C_2 > 0$,  such that
$$
 (\nabla_x \Delta_x^{-1} \varrho_f, j[\LLL f])_{L^2}   \ge 
 \frac12 \| \varrho_f \|_{L^2}^2 
- C_2 \| f^\perp \|_{\HH}^2
 - C_2    \| \iota \DDD^\perp \gamma_{+} f \|_{\partial \HH_+}^2, 
 \quad \forall \, f \in \HH.
  $$
\end{lem}

\begin{proof}[Proof of Lemma~\ref{lem:mass1}] We repeat the proof of \cite[Lemma 3.8]{MR4581432}. 
Writing 
$$
j[\LLL f] = j[-v\cdot \nabla_x f] -  j[f^\perp] 
$$
where we have observed that  $\CCC \pi f = 0$ and $j[\CCC g] = j[g]$, 
and denoting $u :=  (-\Delta_x)^{-1} \varrho_f$, we have 
$$
(-\nabla_x u,  j [\LLL f]  )_{L^2}
= \bigl( \partial_{x_i} u, \partial_{x_j} \int_{\R^d} v_i v_j f \, dv  \bigr)_{L^2} + \bigl( \nabla_x u, j[f^\perp]  \bigr)_{L^2}.
$$
On the one hand,  using the Green formula, we may write 
$$
\bigl( \partial_{x_i}  u , \partial_{x_j} \int_{\R^d} v_i v_j f \, dv  \bigr)_{L^2}  = A + B,
$$
with 
$$
A := - \bigl( \partial_{x_j} \partial_{x_i} u ,  \int_{\R^d} v_i v_j f \, dv  \bigr)_{L^2}, 
\quad
B :=  \int_{\partial\Omega} \partial_{x_i} u\,  n_j(x)  \left( \int_{\R^d} v_i v_j \, \gamma f \, dv \right)  d\sigma_{\! x} . 
$$
Thanks to the decomposition \eqref{eq:fdecomposition}, 
we get
$$
\int_{\R^d} v_i v_j f \, dv
= \delta_{ij} \varrho_f + \int_{\R^d} v_i v_j f^\perp \, dv,
$$
and hence 
$$
\begin{aligned}
A 
&= ( -\Delta_x u , \varrho_f )_{L^2} 
- ( \partial_{x_j} \partial_{x_i}  u , \int_{\R^d} v_i v_j f^\perp \, dv )_{L^2}\\
&= \|\varrho\|^2_{L^2}  
- ( \partial_{x_j} \partial_{x_i}  u , \int_{\R^d} v_i v_j f^\perp \, dv )_{L^2},
\end{aligned}
$$
since $-\Delta_x u = \varrho$ by definition of $u$. Using~\eqref{eq:PoissonNH2}, we have
$$
\begin{aligned}
\left| ( \partial_{x_j} \partial_{x_i}  u , \int_{\R^d} v_i v_j f^\perp \, dv )_{L^2}  \right|
&\lesssim \| D_x^2 u \|_{L^2} \| f^\perp \|_{\HH} \\
&\lesssim \| \varrho_f \|_{L^2} \| f^\perp \|_{\HH} ,
\end{aligned}
$$
from what it follows, thanks to Young's inequality,
$$
A \ge \frac34  \| \varrho_f \|_{L^2}^2 - C \| f^\perp \|_{\HH}^2.  
$$
We now investigate the boundary term $B$. Thanks to Lemma~\ref{lem:boundary}, we have
$$
\begin{aligned}
B &= \int_{\Sigma} \nabla_x u \cdot v \, \gamma f \, n_x \cdot v \, dv \, d\sigma_{\! x} \\
&= \int_{\Sigma_{+}} \nabla_x u\cdot v \iota(x) \DDD^\perp \gamma_{ +} f   \, n_x \cdot v \, dv \, d\sigma_{\! x} \\
&\quad    
+ \int_{\Sigma_{+}} \nabla_x u \cdot [v - \VV_x v] (1-\iota(x)) \DDD^\perp \gamma_{ +} f   \, n_x \cdot v \, dv \, d\sigma_{\! x} \\
&\quad    
+ \int_{\Sigma_{+}} \nabla_x u \cdot [v - \VV_x v] D \gamma_{ +} f   \, n_x \cdot v \, dv \, d\sigma_{\! x} \\
&=: B_1 + B_2 + B_3,
\end{aligned}
$$
and we remark that
$$
v - \VV_x v = 2 n_x (n_x \cdot v),
$$
so that 
$$
\nabla_x u \cdot [v - \VV_x v] = 2 \nabla_x u \cdot n_x \, (n_x \cdot v).
$$
We therefore obtain~$B_2 = B_3 = 0$ thanks to the boundary condition satisfied by $u$ in~\eqref{eq:Poisson-PDEFormulation}.
On the other hand,  the Cauchy-Schwarz inequality and~\eqref{eq:PoissonNH2} yield 
$$
\begin{aligned}
|B_1|
&\le  \| \nabla_x u \|_{L^2_x(\partial\Omega)} \| v \mu \|_{L^1}^{1/2} \| \iota \DDD^\perp \gamma_{+} f \|_{\partial \HH_+} \\
&\lesssim \| \varrho_f \|_{L^2}  \| \iota \DDD^\perp \gamma_{+} f \|_{\partial \HH_+} .
\end{aligned}
$$
Similarly as for the term $A$, we last have 
$$
\bigl| \bigl( \nabla_x u, j[f^\perp]  \bigr)_{L^2} \bigr| \le \| \nabla_x u \|_{L^2} \| j[f^\perp] \|_{L^2}
\lesssim \| \varrho_f \|_{L^2} \| f^\perp \|_{\HH} ,
$$
where we have used the estimate \eqref{eq:PoissonNH2} and twice the Cauchy-Schwarz inequality. 
The proof is then complete by gathering all the previous estimates and by using Young's inequality.  
\end{proof}

\medskip
We finally estimate the third term involved in the Dirichlet form $D$.

\begin{lem}\label{lem:mass2}
There is a constant  $C_3 >0$ such that 
$$
(\nabla_x  \Delta_x^{-1} \varrho[\LLL f] ,  j_f )_{L^2}\\
\ge -  C_3 \| f^\perp \|_{\HH}^2 
$$
\end{lem}

\begin{proof}[Proof of Lemma~\ref{lem:mass2}] From \eqref{def:LLL}, \eqref{def:CCCf} and $\varrho[\CCC f] = 0$, 
one has 
$$
\begin{aligned}
\varrho[\LLL f] 
= \varrho[- v \cdot \nabla_x f] 
= - \Div_x   \int_{\R^d} v f \, dv =  - \Div_x j_f.  
\end{aligned}
$$
On the other hand, we also classically observe 
\bean
j_f  \cdot n_x 
&=& \int_{\R^d} \gamma f \, v \cdot n_x dv 
\\
&=& \iota \Bigl\{ \int_{\Sigma^x_+} \gamma_{+} \!  f \, v \cdot n_x dv -  \int_{\Sigma^x_-} \MMM(v) \widetilde{\gamma_{+} \!  f}  \, |v \cdot n_x| dv  \Bigr\}
\\
&&+ (1-\iota) \Bigl\{ \int_{\Sigma^x_+}  \gamma_{+} \!  f \, v \cdot n_x dv -  \int_{\Sigma^x_-}  \gamma_{+} \!  f \circ \VV_x \, |v \cdot n_x| dv  \Bigr\},
\eean
and using the very definition of $ \widetilde{\gamma_{+} \!  f} $ and $\MMM$ in \eqref{eq:FPK-def_D} and \eqref{eq:FPK-def_M} in the second integral and  the change of variables $v \mapsto \VV_x v$ in the last integral, 
we see that both contributions vanish and we thus  obtain the zero flux condition 
\beqn\label{eq:jn=0}
j_f  \cdot n_x  = 0.
\eeqn
Now let us define 
$$
u := (- \Delta_x)^{-1} \varrho[\LLL f] = (- \Delta_x)^{-1} ( - \Div_x  j_f)
$$
 the unique variational solution to \eqref{eq:Poisson-PDEFormulation} with Neumann boundary condition associated to the source term $\xi = \varrho[\LLL f] = \Div \eta_2$, $\eta_2 := - j_f$. 
From the variational formulation \eqref{eq:Poisson-VarFormulation}, we have 
$$
\begin{aligned}
\| \nabla_x u \|_{L^2}^2 
& = - \int_{\Omega} (\nabla_x \cdot j_f) \, u \, dx \\
&= \int_{\Omega}  j_f \cdot \nabla_x u \, dx 
- \int_{\partial\Omega}  j_f \cdot n_x \, u \, d\sigma_{\! x}
= \int_{\Omega}  j_f \cdot \nabla_x u \, dx  , 
\end{aligned}
$$
where we have used the Green formula and finally \eqref{eq:jn=0} in order to obtain the last equality.  
We deduce 
$$
\| \nabla_x u \|_{L^2}
 \lesssim \| j_f \|_{L^2}.
$$
  thanks to the Cauchy-Schwarz inequality, and thus 
$$
\begin{aligned}
\left| ( -\nabla_x u,  j_f  )_{L^2}   \right| 
&\lesssim \|  \nabla_x u \|_{L^2}   \| j_f \|_{L^2}
\lesssim  \| j_f\|_{L^2}^2.
\end{aligned}
$$
We conclude thanks to \eqref{eq:fdecompositionNormEstim}. 
\end{proof}

\medskip
We are now able to conclude the proof of Theorem \ref{theo:hypo}.

\begin{proof}[Proof of Theorem \ref{theo:hypo}]
Let $f$ satisfy the assumptions of Theorem~\ref{theo:hypo}. Observing that $\sqrt{\iota(2-\iota)} \ge \iota$ since $\iota$ takes values in $[0,1]$, and gathering Lemmas~\ref{lem:micro},~\ref{lem:mass1} and ~\ref{lem:mass2}, one has 
$$
\begin{aligned}
(\!( - \LLL f , f )\!) 
&\ge   \| f^\perp \|_{\HH}^2 + \frac12 \| \sqrt{\iota(2-\iota)} \DDD^\perp \gamma_{ +} f  \|_{\partial \HH_+}^2 \\
&\quad 
+\eps \Big( \frac12 \| \varrho_f \|_{L^2}^2 
- (C_2+C_3) \| f^\perp \|_{\HH}^2 
- C_2 \| \sqrt{\iota(2-\iota)} \DDD^\perp \gamma_{ +} f  \|_{\partial \HH_+}^2  \Big).
\end{aligned}
$$
 Choosing $0 < \eps < 1$ small enough, we get
$$
\begin{aligned}
(\!( - \LLL f , f )\!) 
&\ge \kappa \left( \| f^\perp \|_{\HH}^2 
+\| \varrho_f \|_{L^2}^2 
   \right)  
+ \kappa' \| \sqrt{\iota(2-\iota)} \DDD^\perp \gamma_{ +} f  \|_{\partial \HH_+}^2
\end{aligned}
$$
for some constants $\kappa,\kappa' >0$.
We thus obtain \eqref{eq:hypocoercivityL2} by using the identity \eqref{eq:fdecompositionNorm} and the equivalence \eqref{eq:equivNorm} of the norms $\| \cdot \|_{\HH}$ and $\Nt \cdot \Nt$. 
\end{proof}

\section{Asymptotic behavior: Proof of Theorem~\ref{th:LimitInfty}}
 \label{sec:proofTh3}
 
We repeat the proof of \cite[Theorem~3.1]{MR3779780} and \cite[Theorem~1.4]{MR3488535}, so that we just sketch the arguments.

\begin{proof}[Proof of Theorem~\ref{th:LimitInfty}]
We introduce the splitting 
$$
\AA f := M \chi_R(v) f, \quad \BB := \LLL - \AA,
$$
with $\chi_R(v) := \chi(v/R)$, $\chi \in \DD(\R^d)$, ${\mathbf 1}_{B_1} \le \chi \le {\mathbf 1}_{B_2}$, and 
some constants $M,R>0$ to be fixed below. We denote by $S_\BB$ the semigroup associated to the modified KFP equation associated to the partial differential operator $\BB$ and the same 
reflection condition \eqref{eq:KolmoBdyCond}. 
We define 
\beqn\label{def:frakW2}
\mathfrak W_2 := \left\{ \omega \in \mathfrak W_0 \,; \, \sup_{p \in [1,\infty]} \limsup_{|v| \to \infty} \varpi_{\omega,p} =: \kappa^* < - 1 \right\}, 
\eeqn
where we recall that $ \varpi_{\omega,p}$ is defined in \eqref{def:varpi}. In particular, $\omega := \langle v \rangle^k e^{\zeta |v|^s} \in \mathfrak W_2$ if $s=2$ and $\zeta \in (0,1/2)$, or if $s \in [0,2) $, or if $s=0$ and $k > d+1$. By repeating the proof of Proposition~\ref{prop:EstimLp}, for any $\kappa > \kappa^*$, we may find $M,R > 0$ large enough such that for any $\omega \in \mathfrak W_2$, we have 
$$
\sup_{p \in [1,\infty]} \sup_{v \in \R^d} (\varpi_{\omega,p}(v)  - M \chi_R(v)) \le (\kappa^*+\kappa)/2, 
$$
and thus there exists a constant $C = C(\omega)>0$ such that 
\beqn\label{eq:prop-estimSB}
\| S_\BB(t) f_0 \|_{L^p_\omega} \le C e^{\kappa t} \|   f_0 \|_{L^p_\omega}, \quad \forall \, t \ge 0, 
\eeqn
for any $f_0 \in L^p_\omega$, $1 \le p \le \infty$. 

We now fix two weight functions $\omega_0 = e^{\zeta |v|^2}$ and $\omega_0' = e^{\zeta' |v|^2}$ with $0<\zeta'<\zeta<1/2$ satisfying the conditions of  Theorem~\ref{th:DGNML1Linfty}. By repeating the proof of Theorem~\ref{th:DGNML1Linfty}, we also have 
\beqn\label{eq:th1-estimSB}
\| S_\BB (t) f_0 \|_{L^\infty_{\omega_0'}} \le C  \frac{e^{\kappa t} }{ t^{\nu}}   \| f_0  \|_{L^p_{\omega_0}}, \quad \forall \, t > 0.
\eeqn
Recalling the definition of total mass $ \la\!\la \cdot \ra\!\ra$ in \eqref{eq:intro_mass_cons}, we define 
$$
\Pi g := g - \langle \! \langle g \rangle \! \rangle \mu
$$
and 
$$
\bar S_\LLL := \Pi S_\LLL = S_\LLL  \Pi  = \Pi S_\LLL \Pi.
$$
Iterating the Duhamel formulas
\bean
 S_\LLL &=& S_\BB + S_\BB \AA * \ S_\LLL 
\\
 S_\LLL &=& S_\BB +  S_\LLL  * \AA S_\BB, 
\eean
where $*$ stands the time convolution between operator defined on $\R$ with support on $\R_+$, we deduce that 
\beqn\label{eq:barSL1}
\bar S_\LLL = V_1 \Pi + W_1 * \bar S_\LLL, 
\eeqn
and 
\beqn\label{eq:barSL2}
\bar S_\LLL =  \Pi V_1 + \bar S_\LLL * W_2, 
\eeqn
with 
$$
V_1  :=  \sum_{j=0}^{n-1} (S_\BB \AA)^{*j} * S_\BB, \quad W_1 :=  (S_\BB \AA)^{*n}, \quad W_2 :=  ( \AA S_\BB)^{*n},
$$
where we use the shorthand $U^{*0} := Id$, $U^{*(j+1)} := U * U^{*j} $. Both estimates \eqref{eq:barSL1} and \eqref{eq:barSL2} together, we obtain
\beqn\label{eq:barSLrepresentation}
\bar S_\LLL = V_2 + W_1 * \bar S_\LLL * W_2, 
\eeqn
with 
$$
V_2 :=  V_1 \Pi + W_1 * \Pi V_1. 
$$
For any $\kappa > \kappa^*$ and $n \in \N$, we deduce from \eqref{eq:prop-estimSB} that 
\beqn\label{eq:estimV_2}
\|V_2(t) f_0 \|_{L^p_\omega} \le C e^{\kappa t} \|   f_0 \|_{L^p_\omega}, \quad \forall \, t \ge 0, 
\eeqn
For any $\kappa > \kappa^*$, we deduce from \eqref{eq:prop-estimSB}  and \eqref{eq:th1-estimSB} (see \cite{MR3779780,MR3488535} as well as \cite[Proposition 2.5]{MR3465438}) that  we may find $n \in \N^*$ such that 
\bear\label{eq:estimW1}
\| W_1(t) f_0 \|_{L^p_\omega} &\le& C e^{\kappa t} \|   f_0 \|_{L^2(\mu)}, \quad \forall \, t \ge 0, 
\\ \label{eq:estimW2}
\| W_2(t) f_0 \|_{L^2_{\mu^{-1/2}}} &\le& C e^{\kappa t} \|   f_0 \|_{L^p_\omega}, \quad \forall \, t \ge 0. 
\eear
We also recall that from Theorem~\ref{theo:hypo}, we have  
\beqn\label{eq:barSLestim}
\| \bar S_\LLL f_0 \|_{L^2_{\mu^{-1/2}}} \le C e^{-\lambda t} \| f_0 \|_{L^2_{\mu^{-1/2}}}, \quad \forall \, t \ge 0. 
\eeqn
We conclude to \eqref{eq:th:LimitInfty} by just writing the representation formula \eqref{eq:barSLrepresentation} and using the estimates \eqref{eq:estimV_2}, \eqref{eq:estimW1}, \eqref{eq:estimW2} and \eqref{eq:barSLestim}.
\end{proof}

\medskip\medskip
\paragraph{\textbf{Acknowledgments.}}
The authors warmly thank Clément Mouhot for his PhD course {\it De Giorgi-Nash-Moser theory for kinetic equations} taught at Université Paris Dauphine-PSL during Spring 2023, which has been a source of inspiration for the present work and also the  enlightening discussions  in many occasions about the same subject.  K.C.\ was partially supported by the Project CONVIVIALITY ANR-23-CE40-0003 of the French National Research Agency.

\bigskip

\end{document}